\documentclass[12pt]{amsart}

\usepackage{amssymb,amsbsy}
\usepackage{latexsym,array}
\usepackage{verbatim,color}
\usepackage{fullpage,euscript}
\usepackage{xcolor}
\usepackage{tikz}
\usetikzlibrary{arrows,shapes,trees}
\usepackage[colorlinks=true,linkcolor=blue,urlcolor=my_color,citecolor=magenta]{hyperref}

\definecolor{my_color}{rgb}{0,0.5,0.5}
\definecolor{MIXT}{rgb}{0.8,0.5,0.2}
\definecolor{mixt}{rgb}{0.5,0.3,0.2}
\definecolor{sin}{rgb}{0,0.5,0.5}
\definecolor{darkblue}{rgb}{0,0.1,0.8}
\definecolor{redi}{rgb}{0.5,0,0.4}

\numberwithin{equation}{section}

\input {cyracc.def}
\font\tencyr=wncyr10 
\def\rus{\tencyr\cyracc}

\newtheorem{thm}{Theorem}[section]
\newtheorem{lm}[thm]{Lemma}

\newtheorem{prop}[thm]{Proposition}

\theoremstyle{remark}
\newtheorem{rmk}[thm]{Remark}

\theoremstyle{definition}
\newtheorem{ex}[thm]{Example} 
\newtheorem{df}{Definition}

\newcommand {\ah}{{\mathfrak a}}
\newcommand {\be}{{\mathfrak b}}

\newcommand {\g}{{\mathfrak g}}

\newcommand {\h}{{\mathfrak h}}
\newcommand {\ka}{{\mathfrak k}}
\newcommand {\el}{{\mathfrak l}}
\newcommand {\ma}{{\mathfrak m}}

\newcommand {\p}{{\mathfrak p}}
\newcommand {\q}{{\mathfrak q}}
\newcommand {\rr}{{\mathfrak r}}
\newcommand {\es}{{\mathfrak s}}
\newcommand {\te}{{\mathfrak t}}
\newcommand {\ut}{{\mathfrak u}}

\newcommand {\z}{{\mathfrak z}}

\newcommand {\sln}{{\mathfrak{sl}}_n}

\newcommand {\son}{{\mathfrak {so}}_{n}}

\newcommand{\gt}{\mathfrak}

\newcommand {\eus}{\EuScript}

\newcommand {\gA}{{\eus A}}
\newcommand {\gC}{{\eus C}}

\newcommand {\gS}{{\eus S}}
\newcommand {\gZ}{{\eus Z}}

\newcommand {\ap}{\alpha}

\newcommand {\lb}{\lambda}
\newcommand {\vp}{\varphi}

\newcommand {\ca}{{\mathcal A}}

\newcommand {\gc}{{\mathcal C}}

\newcommand {\co}{{\mathcal O}}

\newcommand {\cz}{{\mathcal Z}}
\newcommand {\BC}{{\mathbb C}}
\newcommand {\BP}{{\mathbb P}}
\newcommand {\BV}{{\mathbb V}}

\newcommand {\BR}{{\mathbb R}}

\newcommand{\id}{{\mathsf{id}}}
\newcommand {\ad}{{\mathrm{ad\,}}}

\newcommand {\Ann}{{\mathsf{Ann\,}}}
\renewcommand {\Re}{{\mathsf{Re\,}}}
\renewcommand {\Im}{{\mathsf{Im\,}}}

\newcommand {\codim}{{\mathrm{codim\,}}}

\newcommand {\ind}{{\mathsf{ind\,}}}
\newcommand {\Lie}{{\mathsf{Lie\,}}}
\newcommand {\Ker}{\operatorname{Ker}}

\newcommand {\rk}{{\mathsf{rk\,}}}
\newcommand {\crk}{{\mathsf{ind}}}

\newcommand {\trdeg}{{\mathrm{tr.deg\,}}}

\newcommand {\tri}{\mathfrak{sl}_2}
\newcommand {\GR}[2]{{\textrm{{\sf\bfseries #1}}}_{#2}}

\newcommand {\ov}{\overline}

\newcommand {\bb}{{\boldsymbol{b}}}

\newcommand {\cW}{{\mathcal W}}

\newcommand {\beq}{\begin{equation}}
\newcommand {\eeq}{\end{equation}}

\newcommand{\curge}{\succcurlyeq}

\renewcommand{\le}{\leqslant}
\renewcommand{\ge}{\geqslant}

\newcommand {\bbk}{\Bbbk}

\begin{document}
\setlength{\parskip}{3pt plus 2pt minus 0pt}
\hfill { {\scriptsize September 11, 2020}}
\vskip1ex

\title[Poisson-commutative subalgebras]
{Compatible Poisson brackets associated with $2$-splittings
and Poisson commutative subalgebras of $\gS(\g)$}
\author[D.\,Panyushev]{Dmitri I. Panyushev}
\address[D.\,Panyushev]%
{Institute for Information Transmission Problems of the R.A.S, Bolshoi Karetnyi per. 19,
Moscow 127051, Russia}
\email{panyushev@iitp.ru}
\author[O.\,Yakimova]{Oksana S.~Yakimova}
\address[O.\,Yakimova]{Universit\"at zu K\"oln,
Mathematisches Institut, Weyertal 86-90, 50931 K\"oln, Deutschland}
\email{oksana.yakimova@uni-jena.de}
\thanks{The first author is supported by RFBR grant {\rus N0} 20-01-00515.
The second author is funded by the Deutsche Forschungsgemeinschaft (DFG, German Research Foundation) --- project number 330450448.}
\maketitle
\begin{abstract}
Let $\gS(\g)$ be the symmetric algebra of a reductive Lie algebra $\g$ equipped with the standard Poisson structure.
If ${\mathcal C}\subset\gS(\g)$ is a Poisson-commutative subalgebra, then $\trdeg\gc\le\bb(\g)$, where
$\bb(\g)=(\dim\g+\rk\g)/2$. We present a method for constructing the Poisson-commutative subalgebra
$\gZ_{\langle\h,\rr\rangle}$  of transcendence degree $\bb(\g)$ via a vector space decomposition
$\g=\h\oplus\rr$ into a sum of two spherical  subalgebras. There are some natural examples, where the algebra $\gZ_{\langle\h,\rr\rangle}$ appears to be polynomial.
The most interesting case is related to the pair $(\be,\ut_-)$, where $\be$ is a Borel subalgebra of $\g$.  Here we prove that
$\gZ_{\langle\be,\ut_-\rangle}$ is maximal Poisson-commutative and is
complete on every regular coadjoint orbit in $\g^*$. Other series of examples are related to involutions
of $\g$.
\end{abstract}

\tableofcontents

\section*{Introduction}

\noindent
The ground field $\bbk$ is algebraically closed and of characteristic $0$. A commutative associative
$\bbk$-algebra ${\ca}$  is a {\it Poisson algebra\/} if there is an additional anticommutative bilinear
operation $\{\,\,,\,\}\!:\,\ca\times \ca \to\ca$ called a {\it Poisson bracket} such that
\[
\begin{array}{cl}
\{a,bc\}=\{a,b\}c+b\{a,c\}, &  \text{(the Leibniz rule)} \\
\{a,\{b,c\}\}+\{b,\{c,a\}\}+\{c,\{a,b\}\}=0 & \text{(the Jacobi identity)}
\end{array}
\]
for all $a,b,c\in{\ca}$. A subalgebra ${\mathcal C}\subset \ca $ is {\it Poisson-commutative} if
$\{\mathcal{C},\mathcal{C}\}=0$. We also write that $\mathcal C$ is a {\sf PC}-{\it subalgebra}.
The {\it Poisson centre} of ${\ca}$ is $\mathcal{ZA}=\{z\in\ca \mid \{z,a\}=0  \ \forall a\in{\ca}\}$.
Two Poisson brackets on $\ca$ are said to be {\it compatible}, if all their linear combinations are again Poisson brackets. Usually, Poisson algebras occur in nature as algebras of functions on varieties
(manifolds), and we  only need the case, where such a variety is the dual of a Lie algebra
$\q$ and hence $\ca=\bbk[\q^*]=\gS(\q)$ is a polynomial ring in $\dim\q$ variables.

There is a general method for constructing a ``large'' Poisson-commutative subalgebra of $\gS(\q)$
associated with a pair of compatible brackets, see e.g.~\cite[Sect.~1]{duzu}. Let $\{\,\,,\,\}'$ and $\{\,\,,\,\}''$
be compatible Poisson brackets on $\q^*$. This yields a two parameter family of Poisson brackets
$a\{\,\,,\,\}'+b\{\,\,,\,\}''$, $a,b\in\bbk$. As we are only interested in the corresponding Poisson centres, it is
convenient to organise this, up to scaling, in a 1-parameter family $\{\,\,,\,\}_t=\{\,\,,\,\}'+t\{\,\,,\,\}''$,
$t\in\BP=\bbk\cup\{\infty\}$, where $t=\infty$ corresponds to the bracket $\{\ ,\ \}''$.
The {\it index\/} $\crk\{\,\,,\,\}$ of a Poisson bracket $\{\,\,,\,\}$ is defined in Section~\ref{sect:prelim}.  For
almost all $t\in\BP$, $\crk\{\,\,,\,\}_t$ has one and the same (minimal) value. Set
$\BP_{\sf reg}=\{t\in \BP\mid \crk\{\,\,,\,\}_t \text{ is minimal}\}$ and
$\BP_{\sf sing}=\BP\setminus \BP_{\sf reg}$. Let $\cz_t$ denote the Poisson centre of
$(\gS(\q),\{\,\,,\,\}_t)$. The crucial fact is that the algebra $\gZ\subset \gS(\q)$ generated by
$\{\cz_t\mid t\in\BP_{\sf reg}\}$ is Poisson-commutative w.r.t{.} to any bracket in the family.
In many cases, this construction provides a {\sf PC}-{subalgebra} of $\gS(\q)$ of maximal transcendence degree.

A notable realisation of this scheme is the {\it argument shift method} of~\cite{mf}. It employs the
Lie--Poisson bracket on $\q^*$ and a Poisson bracket $\{\,\,,\,\}_\gamma$ of degree zero associated
with $\gamma\in\q^*$. Here $\{\xi,\eta\}_\gamma=\gamma([\xi,\eta])$ for
$\xi,\eta\in\q$. The algebras $\gZ=\gZ_\gamma$ occurring in this approach are known nowadays as
 {\it  Mishchenko--Fomenko subalgebras}.

Let $G$ be a connected semisimple Lie group with $\Lie(G)=\g$. In \cite{oy}, we studied compatible
Poisson brackets and {\sf PC}-subalgebras related to an involution of $\g$.
Our main object now is the $1$-parameter family of linear Poisson brackets on $\g^*$ related to a
{\it $2$-splitting} of $\g$,  i.e., a vector space sum $\g=\h\oplus\rr$, where $\h$ and $\rr$ are
Lie subalgebras. This is also the point of departure for the Adler--Kostant--Symes theorems and
subsequent results, see~\cite[Sect.~4.4]{AMV}, \cite[\S\,2]{t16}. But our further steps are quite different.
For $x\in\g$, let $x_\rr\in\rr$ and $x_{\h}\in\h$ be the components of  $x$.
Here one can contract $\g$ to either $\h\ltimes\rr^{\rm ab}$ or $\rr\ltimes \h^{\rm ab}$.
Let $\{\,\,,\,\}_{0}$ and $\{\,\,,\,\}_{\infty}$ be the corresponding Poisson brackets on $\g^*$. Then
\[
\{x,y\}_0=\begin{cases}
[x,y] & \text{ if } \ x,y\in \h, \\
{}[x,y]_{\rr} & \text{ if } \ x \in\h,y\in \rr, \\
\quad 0 & \text{ if } \ x,y\in \rr ,
\end{cases} \
\text{ and } \
\{x,y\}_\infty=\begin{cases}
[x,y] & \text{ if } \ x,y\in \rr, \\
{}[x,y]_{\h} & \text{ if } \ x \in\h,y\in \rr, \\
\quad 0 & \text{ if } \ x,y\in \h.
\end{cases}
\]
Since $ \{\,\,,\,\}=\{\,\,,\,\}_0+\{\,\,,\,\}_\infty$ is a Poisson bracket, these two brackets are compatible,
cf.~\cite[Lemma~1.1]{oy}. Consider the $1$-parameter family of Poisson brackets
\[
   \{\,\,,\,\}_t=\{\,\,,\,\}_0+t\{\,\,,\,\}_\infty,
\]
where $t\in \BP$.
Here $\bbk^\times\subset \BP_{\sf reg}$. Note that these brackets are different from the bracket
\[
    (x,y)\mapsto  [x_\h,y_\h]-  [x_\rr,y_\rr]
\]
considered in the Adler--Kostant--Symes theory. The algebras $\g_{(0)}=\h\ltimes\rr^{\sf ab}$ and
$\g_{(\infty)}=\rr\ltimes\h^{\sf ab}$  are
{\it In\"on\"u--Wigner contractions} of $\g$, and a lot of information on their symmetric invariants is
obtained in~\cite{contr,Y-imrn}.

Let $\gZ=\gZ_{\langle\h,\rr\rangle}$ denote the subalgebra of $\gS(\g)$ generated by all centres $\cz_t$
with $t\in \BP_{\sf reg}$. Then $\{\gZ,\gZ\} =0$ and therefore
\[
   \trdeg \gZ\le \frac{1}{2}(\dim\g+\rk\g)=\bb(\g).
\]
This upper bound for $\trdeg\gZ$ is attained
if $\ind\{\,\,,\,\}_0=\ind\{\,\,,\,\}_{\infty}=\ind\{\,\,,\,\}=\rk\g$, i.e., $\BP=\BP_{\sf reg}$,
see Theorem~\ref{thm:dim-Z}. A $2$-splitting with such property is said to be {\it non-degenerate.}
We show that the $2$-splitting $\g=\h\oplus\rr$ is non-degenerate if and only if
both subalgebras $\h$ and $\rr$ are {\it spherical}, see Theorem~\ref{thm:c=0} and Remark~\ref{Dima-T}
for details.
Therefore, we concentrate on $2$-splittings involving spherical subalgebras of $\g$.
This allows us to point out many natural pairs $(\h,\rr)$ such that $\trdeg\gZ=\bb(\g)$.
Furthermore, in several important cases, $\gZ$ is a polynomial algebra.

1) \ Consider the $2$-splitting $\g =\be\oplus\ut_-$, where $\be$ and $\be_-$ are two opposite Borel
subalgebras, $\te=\be\cap\be_-$,  and $\ut_-=[\be_-,\be_-]$. The {\sf PC}-subalgebra
$\gZ=\gZ_{\langle\be,\ut_-\rangle}$ has a nice set of algebraically independent generators.
Let $\{H_i \mid 1\le i\le \rk\g \}\subset\gS(\g)^{\g }$ be a set of homogeneous basic invariants and
$d_i=\deg H_i$. The splitting $\g =\gt b\oplus\ut_-$ leads to a bi-grading in
$\gS(\g)$ and the decomposition $H_i=\sum_{j=0}^{d_i}(H_i)_{(j,d_i-j)}$.
Then $\gZ_{\langle\be,\ut_-\rangle}$ is freely generated by
the bi-homogeneous components $(H_i)_{(j,d_i-j)}$ with $1\le i\le \rk\g $,
$1\le j \le d_j-1$ and a basis for the Cartan subalgebra  $\te$, see Theorem~\ref{thm:b-n_polynomial}.

It is easily seen that if ${\mathcal C}\subset\gS(\g)$ is a {\sf PC}-subalgebra and
$\trdeg{\mathcal C}=\bb(\g)$, then $\mathcal C$ is complete on generic regular $G$-orbits,
cf. Lemma~\ref{obvious}. Using properties of the principal nilpotent orbit in $\g\simeq\g^*$, we are able
to prove that $\gZ_{\langle\be,\ut_-\rangle}$ is {\it complete} on {\bf each} regular coadjoint orbit of $G$
(Theorem~\ref{thm:b-n_complete})  and that it is a {\it maximal\/} {\sf PC}-subalgebra of $\gS(\g)$
(Theorem~\ref{max-u}).  One can also consider a more general setting, where $\be$ is replaced with an
arbitrary parabolic $\p\supset\be$, see Remark~\ref{rem:setting-p}.

2) \  Let $\sigma$ be an involution of maximal rank of $\g$, i.e., the $(-1)$-eigenspace of $\sigma$,
$\g_1$, contains a Cartan subalgebra of $\g$.  If $\g_0$ is the corresponding fixed-point subalgebra, then
there is a Borel $\be$ such that $\g=\be\oplus\g_0$. This $2$-splitting is non-degenerate and we show
that $\gZ_{\langle\be,\g_0\rangle}$ is a polynomial algebra, see Theorem~\ref{thm:g0-b-polynomial}. At
least for $\g=\sln$, this {\sf PC}-subalgebra is also maximal (Example~\ref{ex:sl-so}). It is likely that the
maximality takes place for all simple $\g$.
\\ \indent
More generally, a non-degenerate $2$-splitting is associated with any involution $\sigma$ such that
$\g_1\cap\g_{\sf reg}\ne \varnothing$, see Remark~\ref{rmk:any-invol}.

3) \ Consider a semisimple Lie algebra $\tilde\g=\g\times\g$ and involution $\tau$ that permutes the
summands. Here $\tilde\g_1\cap\tilde\g_{\sf reg} \ne \varnothing$ and this yields a natural
non-degenerate $2$-splitting $\g\times\g=\Delta_\g\oplus\h$, which represents the famous Manin
triple. The corresponding
{\sf PC}-subalgebra $\gZ\subset\gS(\g\oplus\g)$ appears to be polynomial. This has a well-known
counterpart over $\BR$ that involves a compact real form $\ka$ of $\g$.
Namely, if  $\g$ is considered as a real Lie algebra, then it has the {\it Iwasawa decomposition}
$\g=\ka\oplus\rr$~\cite[Ch.\,5,\,\S 4]{t41},
where $\rr\subset \be$ is a solvable real Lie algebra. We prove that the $\BR$-algebra
$\gZ_{\langle\ka,\rr\rangle}$ is also polynomial, see Section~\ref{sect:k-b}.

We refer to \cite{duzu} for generalities on Poisson varieties, Poisson tensors, symplectic leaves, etc.
Our general reference for algebraic groups and Lie algebras is~\cite{t41}.

\section{Preliminaries on the coadjoint representation}
\label{sect:prelim}

\noindent
Let $Q$\/ be a connected linear algebraic group with $\Lie(Q)=\q$. Then $\gS_\bbk(\q)=\gS(\q)$ is the
symmetric algebra of $\q$ over $\bbk$. It is identified with the graded algebra of polynomial functions on
$\q^*$, and we also write $\bbk[\q^*]$ for it.
\\ \indent
Write $\q^\xi$ for the {\it stabiliser\/} in $\q$ of $\xi\in\q^*$. The {\it index of}\/ $\q$, $\ind\q$, is the minimal codimension of $Q$-orbits in $\q^*$. Equivalently,
$\ind\q=\min_{\xi\in\q^*} \dim \q^\xi$. Let $\bbk(\q^*)^Q$ be the field of $Q$-invariant rational functions
and $\bbk[\q^*]^Q$ the algebra of $Q$-invariant polynomial functions on $\q^*$.
By the Rosenlicht theorem, 
one has $\ind\q=\trdeg\bbk(\q^*)^Q$. Therefore $\trdeg\bbk[\q^*]^Q\le \ind\q$.
The ``magic number'' associated with $\q$ is $\bb(\q)=(\dim\q+\ind\q)/2$.
Since the coadjoint orbits are even-dimensional, the magic number is an integer. If $\q$ is reductive, then
$\ind\q=\rk\q$ and $\bb(\q)$ equals the dimension of a Borel subalgebra. The Lie--Poisson bracket on
$\bbk[\q^*]$ is defined on the elements of degree $1$ (i.e., on $\q$) by $\{x,y\} :=[x,y]$.
The {\it Poisson centre\/} of $\gS(\q)$ is
\[
    \cz\gS(\q)=\{H\in \gS(\q)\mid \{H,x\} =0 \ \ \forall x\in\q\}=\gS(\q)^\q .
\]
As $Q$ is connected, we have $\gS(\q)^\q=\gS(\q)^{Q}=\bbk[\q^*]^Q$.
The set of $Q$-{\it regular\/} elements of $\q^*$ is
\beq       \label{eq:regul-set}
    \q^*_{\sf reg}=\{\eta\in\q^*\mid \dim \q^\eta=\ind\q\} .
\eeq
The $Q$-orbits in $\q^*_{\sf reg}$ are also called {\it regular}.
Set $\q^*_{\sf sing}=\q^*\setminus \q^*_{\sf reg}$.
We say that $\q$ has the {\sl codim}--$n$ property if $\codim \q^*_{\sf sing}\ge n$.  By~\cite{ko63}, the semisimple algebras $\g$ have the {\sl codim}--$3$ property.

Let $\Omega^i$ be the $\gS(\q)$-module of differential $i$-forms on $\q^*$. Then
$\Omega=\bigoplus_{i=0}^n \Omega^i$ is the $\gS(\q)$-algebra of regular
differential forms on $\q^*$. Likewise,
$\cW=\bigoplus_{i=0}^n \cW^i$  is the graded skew-symmetric algebra
of polyvector fields, which is generated by the $\gS(\q)$-module $\cW^1$ of polynomial vector fields on $\q^*$.
Both algebras are free $\gS(\q)$-modules.
The {\it Poisson tensor (bivector)\/} $\pi\in \operatorname{Hom}_{\gS(\q)}(\Omega^2,{\gS(\q)})$
associated with a Poisson bracket $\{\,\,,\,\}$ on $\q^*$ is defined by the equality
$\pi(\textsl{d}f\wedge \textsl{d}g)=\{f,g\}$ for $f,g\in \gS(\q)$. For any $\xi\in\q^*$, $\pi(\xi)$ defines a
skew-symmetric bilinear form on $T^*_\xi(\q^*)\simeq\q$.
Formally, if $f,g\in\gS(\q)$, $v=\textsl{d}_\xi f$, and $u=\textsl{d}_\xi g$,  then
$\pi(\xi)(v,u)=\pi(\textsl{d}f\wedge \textsl{d}g)(\xi)=\{f,g\}(\xi)$.
In view of the duality between differential 1-forms and vector
fields, we may regard $\pi$ as an element of $\cW^2$. Let $[[\ ,\ ]]: \cW^i\times \cW^j \to \cW^{i+j-1}$
be the Schouten bracket. The Jacobi identity for $\pi$ is equivalent to that
$[[\pi,\pi]]=0$, see e.g.~\cite[Chapter\,1.8]{duzu}.

\begin{df} \label{def-crk}
The {\it index\/} of a Poisson bracket $\{\,\,,\,\}$ on $\q^*$, denoted $\ind\{\,\,,\,\}$, is the
minimal codimension of the symplectic leaves in $\q^*$. 
\end{df}
It is easily seen that if $\pi$ is the corresponding Poisson tensor, then
\\[.4ex]
\centerline{$\crk\{\,\,,\,\}=\min_{\xi\in\q^*} \dim  \ker \pi(\xi)=n-\max_{\xi\in\q^*}\rk \pi(\xi)$.}
\\
Recall that for a Lie algebra $\q$ and the dual space $\q^*$ equipped with the Lie--Poisson bracket
$\{\,\,,\,\} $, the symplectic leaves are the coadjoint $Q$-orbits. Hence
$\crk\{\,\,,\,\} =\ind\q$.

\subsection{Complete integrability on coadjoint orbits}
For $\xi\in\q^*$, let $Q{\cdot}\xi$ denote its coadjoint  $Q$-orbit. If
$\psi_\xi\!: T^*_\xi \q^* \to T^*_\xi(Q{\cdot}\xi)$ is the canonical projection, then $\ker\psi_\xi=\q^\xi$.
Let $\pi$ be the Poisson tensor of the Lie--Poisson bracket on $\q^*$.
Then $\pi(\xi)(x,y)=\xi([x,y])$ for $x,y\in\q$.
The skew-symmetric form $\pi(\xi)$ is non-degenerate on
$T^*_\xi(Q{\cdot}\xi)$. The algebra $\bbk[Q{\cdot}\xi]$ carries the Poisson structure, which is inherited from $\q^*$.
We have
 $$\{F_1|_{Q{\cdot}\xi},F_2|_{Q{\cdot}\xi}\}=\{F_1,F_2\}|_{Q{\cdot}\xi}$$
 for all $F_1,F_2\in\gS(\q)$.
The coadjoint orbit $Q{\cdot}\xi$ is a smooth symplectic variety.

\begin{df}     \label{com-fam}
A set $\boldsymbol{F}=\{F_1,\ldots,F_m\}\subset \bbk[Q{\cdot}\xi]$ is said to be {\it a complete family in involution} if
$F_1,\ldots,F_m$ are algebraically independent, $\{F_i,F_j\}=0$ for all $i,j$, and
$m=\frac{1}{2}\dim (Q{\cdot}\xi)$. In the terminology of \cite[Def.~4.13]{AMV}, here
$(Q{\cdot}\gamma, \{\,\,,\,\}, \boldsymbol{F})$ is a {\it completely integrable system}.
\end{df}

The interest in integrable systems arose from the theory of differential equations and in particular
equations of motions, see e.g. \cite[Chapter~4]{AMV}. By now this theory has penetrated
nearly all of mathematics and has had a definite impact on such remote fields as
combinatorics and number theory. A rich interplay between Lie theory and complete integrability
is well-documented, see~\cite{t16,AMV,Per}.
Applications of {\sf PC}-subalgebras of $\gS(\q)$ are one of the striking examples of this interplay.

Let $\gA\subset \gS(\q)$ be a Poisson-commutative subalgebra. Then the restriction of $\gA$ to
$Q{\cdot}\xi$, denoted $\gA|_{Q{\cdot}\xi}$,
is Poisson-commutative for every $\xi$.  We say that $\gA$ is {\it complete on\/} $Q{\cdot}\xi$, if
$\gA|_{Q{\cdot}\xi}$ contains a complete family in involution.
The condition is equivalent to the equality $\trdeg (\gA|_{Q{\cdot}\xi}) = \frac{1}{2}\dim (Q{\cdot}\xi)$.

\begin{lm}      \label{obvious}
Suppose that $\gA\subset \gS(\q)$ is Poisson-commutative, $\xi\in\q^*_{\sf reg}$, and
$\dim\textsl{d}_\xi \gA=\bb(\q)$. Then $\gA$ is complete on $Q{\cdot}\xi$.
\end{lm}
\begin{proof}
Since $\xi$ is regular, we have $\dim\ker\psi_\xi=\ind\q$.
Therefore
\[
    \dim \psi_\xi(\textsl{d}_\xi \gA) \ge \bb(\q)-\ind\q=\frac{1}{2}\dim (Q{\cdot}\xi)
\]
as required.
\end{proof}

\section{In\"on\"u--Wigner contractions and their invariants}
\label{sect:2}
Let  $\h$ be a Lie subalgebra of $\q$. Choose a complementary subspace $V$ to $\h$ in $\q$, so that
$\q=\h \oplus V$ is a vector space decomposition.
For any $s\in\bbk^{\times}$, define the invertible linear map $\vp_s\!: \q\to\q$ by
setting $\vp_s\vert_{\h}=\id$, $\vp_s\vert_{V}=s{\cdot}\id$. Then $\vp_s\vp_{s'}=\vp_{ss'}$ and
$\vp_s^{-1}=\vp_{s^{-1}}$, i.e., this yields a one-parameter subgroup of ${\rm GL}(\q)$. The
map $\vp_s$ defines a new (isomorphic to the initial) Lie algebra structure $[\,\,,\,]_{(s)}$
on the same vector space $\q$ by the formula
\beq     \label{eq:fi_s}
    [x,y]_{(s)}=\vp_s^{-1}([\vp_s(x),\vp_s(y)]).
\eeq
The corresponding Poisson bracket is $\{\ ,\ \}_{(s)}$. We naturally extend $\vp_s$ to an automorphism of
$\gS(\q)$.  Then the centre of the Poisson algebra $(\gS(\q),\{\,\,,\,\}_s)$ equals
$\vp_s^{-1}(\gS(\q)^{\q})$.

The condition $[\h,\h]\subset \h$ implies that there is a limit of the brackets $[\ ,\ ]_{(s)}$
as $s$ tends to zero. The limit bracket is denoted by $[\ ,\ ]_{(0)}$ and the corresponding Lie algebra
$\q_{(0)}$ is the semi-direct product  $\h\ltimes V^{\sf ab}$, where $V^{\sf ab}\simeq \q/\h$ as an $\h$-module and $[V^{\sf ab},V^{\sf ab}]_{(0)}=0$. More precisely, if $x=h+v\in \q_{(0)}$ with $h\in\h$ and $v\in V$, then
\[
   [h+v,h'+v']_{(0)}=[h,h']+ [h,v']_V- [h',v]_V ,
\]
where $z_V$ denotes the $V$-component of $z\in\q_{(0)}$. The limit algebra $\q_{(0)}$ is called an
{\it In\"on\"u-Wigner} (=\,{\sf IW}) or {\it one-parameter contraction\/} of $\q$, see \cite[Ch.\,7,\S\,2.5]{t41}
or~\cite[Sect.\,1]{alafe}. Below, we will repeatedly use the following

{\bf Independence principle.}
{\it The {\sf IW}-contraction $\q_{(0)}$
does not depend on the initial choice of a complementary subspace $V$.}
\\
Therefore,
when there is no preferred choice of $V$, we write $\q_{(0)}=\h\ltimes (\q/\h)^{\sf ab}$.
By a general property of Lie algebra contractions,
we have $\ind\q_{(0)}\ge \ind\q$. We need conditions on $\q$ and $\h$ under which the index of the
{\sf IW}-contraction does not increase. For this reason, we switch below to the case in which
$\q=\g$ is reductive and hence $G$ is a connected reductive algebraic group.

For any irreducible algebraic $G$-variety $X$, there is the notion of the {\it complexity} of $X$,
denoted $c_G(X)$, see~\cite{vi86}. Namely, $c_G(X)=\dim X-\max_{x\in X}\dim B{\cdot}x$, where
$B\subset G$ is a Borel subgroup. 
Then $X$ is said to be {\it spherical}, if $c_G(X)=0$, i.e., if $B$ has a dense orbit in $X$. In particular,
for any subgroup $H\subset G$, one can consider the complexity of the homogeneous space $X=G/H$.
Then $H$ (or $\h=\Lie(H)$) is said to be {\it spherical\/} if $c_G(G/H)=0$.

\begin{thm}    \label{thm:c=0}
Suppose that $G$ is reductive and the homogeneous space $G/H$ is quasi-affine. Then
$\ind (\h\ltimes (\g/\h)^{\sf ab})=\ind \g+ 2c_G(G/H)=\rk \g+ 2c_G(G/H)$. In particular,
$\ind (\h\ltimes (\g/\h)^{\sf ab})=\ind \g$ if and only if\/ $\h$ is a spherical subalgebra of\/ $\g$.
\end{thm}
\begin{proof}
For the affine homogeneous spaces, a proof is given in \cite[Prop.\,9.3]{p05}. Here we demonstrate that
that proof actually applies in the general quasi-affine setting.

Let $\h^\perp$ be the annihilator of $\h$ in the dual space $\g^*$. It is an $H$-submodule of $\g^*$
that is called the {\it coisotropy representation\/} of $H$. Here $\h^\perp$ and $\g/\h$ are dual
$\h$-modules. (If $\Phi$ is a $G$-invariant bilinear form on $\g$, then one can identify $\g^*$ and $\g$
using $\Phi$,
and consider $\h^\perp$ as a subspace of $\g$.)

Let $\bbk(\h^\perp)^H$ denote the subfield of $H$-invariants in $\bbk(\h^\perp)$. The Ra\"is formula
for the index of semi-direct products~\cite{rais} asserts that
$\ind(\h\ltimes (\g/\h)^{\sf ab})=\trdeg\bbk(\h^\perp)^H + \ind \es$,
where $\es$ is the $\h$-stabiliser of a generic point in $\h^\perp$. (Here we use the fact that
$\g/\h$ and $\h^{\perp}$ are dual $H$-modules.)
Since $G/H$ is quasi-affine, $\es$ is reductive \cite[Theorem\,2.2.6]{these-p}.
Hence $\ind (\h\ltimes (\g/\h)^{\sf ab})=\trdeg\bbk(\h^\perp)^H+\rk\es$. Moreover, there is a formula for
$c_G(G/H)$ in terms of the action $(H:\h^\perp)$. Namely,
$2c_G(G/H)=\trdeg\bbk(\h^\perp)^H- \rk\g+ \rk\es$ \cite[Cor.\,2.2.9]{these-p}. Whence the conclusion.
\end{proof}

\begin{rmk}    \label{rem:parab-contr}
If $P\subset G$ is a parabolic subgroup, then $G/P$ is {\bf not} quasi-affine. However, it is proved
in~\cite[Theorem\,4.1]{alafe2} that $\ind(\p\ltimes (\g/\p)^{\sf ab})=\rk\g$. For a  Borel subgroup $B$,
this appeared already in~\cite[Cor.\,3.5]{alafe}. The reason is that the Ra\"\i s formula readily implies
that $\ind(\p\ltimes (\g/\p)^{\sf ab})=\ind\g^e$, where $e\in \p^{\sf nil}$ is a Richardson element and
$\p^{\sf nil}$ is the nilradical of $\p$.
The famous {\it Elashvili conjecture} asserts that $\ind\g^e=\rk\g$ for any $e\in \g$.
For the Richardson elements, a conceptual proof of the Elashvili conjecture is given in \cite{CM}.
\end{rmk}

\begin{rmk}    \label{Dima-T}
In an earlier version of this article, we conjectured that $\ind (\h\ltimes (\g/\h)^{\sf ab})=\ind \g$ \ for
{\it\bfseries any\/} spherical subalgebra $\h$. Having heard from us about this problem,
D.\,Timashev
informed us that combining
some results of Knop~\cite{knop}, the Elashvili conjecture, and the scheme of proof of
Theorem~\ref{thm:c=0}, one can extend Theorem~\ref{thm:c=0} to
{\bf arbitrary} homogeneous spaces $G/H$. This general argument is outlined below. We are grateful to
Timashev for providing necessary details.

Let $\mathsf T^*(G/H)=G\times^H\h^\perp$ be the cotangent bundle of $G/H$. The generic stabiliser
for the $H$-action on $\h^\perp$ is isomorphic to a generic stabiliser for the $G$-action on
$\mathsf T^*(G/H)$. Let $\es$ be such a stabiliser. (If $G/H$ is quasi-affine, then $\es$ is
reductive. But this is not so in general.) A general description of $\es$, see~\cite[Sect.\,8]{knop}, can be
stated as follows. Fix a maximal torus $T\subset B$ and set $\te=\Lie(T)$.
For a generic $B$-orbit $\co$ in $G/H$, consider the parabolic subgroup $P=\{g\in G\mid g(\co)\subset\co\}\supset B$. Let $\Gamma\in \mathfrak X(T)$ be
the lattice of weights of all $B$-semi-invariants in the field $\bbk(G/H)$ and $\te_0$ the Lie algebra of
$\Ker\Gamma\subset T$. The rank of $\Gamma$ is called the {\it rank\/} of $G/H$, denoted $r_G(G/H)$.
Set $\ah=\te_0^\perp$, the orthogonal complement w.r.t. $\Phi\vert_\te$ and
consider the Levi subgroup $M=Z_G(\ah)\subset G$. The weights in $\Gamma$ can be regarded as
characters of $M$ and we consider the identity component of their common kernel as a subgroup of
$M$, denoted $M_0$. Clearly, $M_0$ is reductive. Write $\ma_0\subset \ma$ for their Lie algebras.
\\ \indent
Let $P_-$ be the opposite to $P$ parabolic subgroup and $\p_-^{\sf nil}$  the nilradical
of $\p_-=\Lie(P_-)$. Then $M_0\cap P_-$ is a parabolic subgroup of
$M_0$ and Knop's description boils down to the assertion that
$\es$ is the generic stabiliser for the linear action of
$M_0\cap P_-$ on $\ma\cap \p_-^{\sf nil}=\ma_0\cap \p_-^{\sf nil}$.
It is noticed by Timashev that $\es\subset\ma_0$ is actually the stabiliser in $\ma_0$ of a Richardson
element in $\ma_0\cap\p_-^{\sf nil}=(\ma_0\cap\p_-)^{\sf nil}$. Hence $\ind\es=\rk M_0$ by the Elashvili conjecture. By the Ra\"is formula, we have
\beq         \label{eq:rais}
     \ind(\h\ltimes (\g/\h)^{\sf ab})=\trdeg\bbk(\h^\perp)^H + \ind \es .
\eeq
The general theory developed in~\cite{knop,p90} implies that
$\trdeg\bbk(\h^\perp)^H=2c_G(G/H)+r_G(G/H)$. 
The last ingredient is that, by the very construction of $M_0$, one has $\rk M_0=\rk\g- r_G(G/H)$. Gathering the above formulae, we obtain
$2c_G(G/H)+\rk\g$ in the right-hand side of~\eqref{eq:rais}.
\end{rmk}

Associated with the vector space sum $\g=\h\oplus V$, one has the bi-homogeneous decomposition
of any homogeneous $H\in \gS(\g)$:
\[
    \textstyle H=\sum_{i=0}^{d} H_{i,d-i} \ ,
\]
where $d=\deg H$ and $H_{i,d-i}\in \gS^i(\h)\otimes \gS^{d-i}(V)\subset \gS^{d}(\g)$. Then $(i,d-i)$ is
the {\it bi-degree\/} of $H_{i,d-i}$.
Let $H^\bullet$ denote the nonzero bi-homogeneous component of $H$ with
maximal $V$-degree. Then $\deg_{V}\! H=\deg_{V} H^\bullet$.
Similarly, $H_{\bullet}$ stands for the nonzero bi-homogeneous component of $H$ with
maximal $\h$-degree, i.e., minimal $V$-degree.

It is known that if $H\in \cz\gS(\g)$, then $H^\bullet\in \cz\gS(\h\ltimes V^{\rm ab})$~\cite[Prop.\,3.1]{coadj}.
However, it is not always the case that $\cz\gS(\h\ltimes V^{\rm ab})$ is generated by the functions of the
form $H^\bullet$ with $H\in \cz\gS(\g)$.
Let $\{H_1,\dots,H_l\}$, $l=\rk\g$,  be a set of homogeneous algebraically independent generators of
$\gS(\g)^\g$ and $d_i=\deg H_i$. Then $\sum_{i=1}^l d_i=\bb(\g)$.

\begin{df}     \label{def:ggs0}
We say that $H_1,\dots,H_l$ is an $\h$-{\it good generating system} in $\gS(\g)^\g$
(=\,$\h$-{\it {\sf g.g.s.}\/} for short) 
if $H_1^\bullet,\dots,H_l^\bullet$ are algebraically independent.
\end{df}

The importance of {\sf g.g.s.} is readily seen in the following fundamental results.

\begin{thm}[{\cite[Theorem\,3.8]{contr}}]    \label{thm:kot14}
Let $H_1,\dots,H_l$ be an arbitrary set of homogeneous algebraically independent generators of\/
$\gS(\g)^\g$ and $\g=\h\oplus V$. Then
\begin{itemize}
\item[\sf (i)] \ $\sum_{j=1}^l \deg_{V}\! H_j\ge \dim V$;
\item[\sf (ii)] \  $H_1,\dots,H_l$ is an\/ $\h$-{\sf g.g.s.} if and only if\/ $\sum_{j=1}^l \deg_V \! H_j=\dim V$.
\end{itemize}
\end{thm}

Furthermore, if the contraction $\g\leadsto \g_{(0)}=\h\ltimes (\g/\h)^{\sf ab}$ has some extra properties,
then the existence of an $\h$-{\sf g.g.s.} provides the generators of
$\cz\gS(\g_{(0)})$. More precisely, Theorem~3.8(iii) in \cite{contr} yields the following:

\begin{thm}    \label{thm:h-ggs+codim2}
Suppose that $\g_{(0)}=\h\ltimes (\g/\h)^{\sf ab}$ has the {\sl codim}--$2$ property and\/
$\ind\g_{(0)}=\ind\g$. If  there is an $\h$-{\sf g.g.s.} $H_1,\dots,H_l$ in $\gS(\g)^\g$,  then
$H_1^\bullet,\dots,H_l^\bullet$ \ freely generate $\gS(\g_{(0)})^{\g_{(0)}}$. In particular,
$\gS(\g_{(0)})^{\g_{(0)}}$ is a polynomial ring.
\end{thm}

\section{$2$-splittings of $\g$ and Poisson-commutative subalgebras}
\label{sect:3}

\noindent
Let $\g$ be a semisimple Lie algebra. The sum $\g=\h\oplus\rr$ is called a {\it $2$-splitting of} $\g$,
if both summands are Lie subalgebras.  Then $\g^*$ acquires the decomposition $\g^*=\h^*\oplus\rr^*$,
where $\rr^*=\Ann\!(\h)=\h^\perp$, $\h^*=\Ann\!(\rr)=\rr^\perp$.
Given a $2$-splitting $\g=\h\oplus\rr$, one can consider two {\sf IW}-contractions. Here either
subalgebra is the preferred complement to the other, so we write
$\h\ltimes\rr^{\sf ab}$ and $\rr\ltimes\h^{\sf ab}$ for these contractions. The important feature of this situation is that the corresponding Poisson-brackets are compatible and their non-trivial linear combinations define Lie algebras isomorphic to $\g$.

If $x=x_{\h}+x_{\rr}\in \g$, then the Lie--Poisson bracket on $\g^*$ is decomposed as follows
\[
\{x,y\}=\underbrace{[x_\h,y_\h]+ [x_{\h},y_{\rr}]_{\rr} + [x_{\rr},y_{\h}]_{\rr}}_{\{x,y\}_{0}} +
\underbrace{[x_{\h},y_{\rr}]_{\h} + [x_{\rr},y_{\h}]_{\h} + [x_{\rr},y_{\rr}]}_{\{x,y\}_{\infty}}.
\]
Here the bracket $\{\,\,,\,\}_{0}$ (resp. $\{\,\,,\,\}_{\infty}$) corresponds to $\g_{(0)}=\h\ltimes\rr^{\sf ab}$
(resp. $\g_{(\infty)}=\rr\ltimes\h^{\sf ab}$).
Using this decomposition, we introduce a $1$-parameter family of Poisson brackets on $\g^*$:
\[
   \{\,\,,\,\}_{t}=\{\,\,,\,\}_{0}+t\{\,\,,\,\}_{\infty},
\]
where $t\in \BP=\bbk\cup\{\infty\}$ and we agree that $\{\,\,,\,\}_{\infty}$ is the Poisson bracket
corresponding to $t=\infty$.
It is easily seen that  $\{\,\,,\,\}_{t}$ with $t\in \bbk^\times$ is given by the map $\vp_t$,
see Eq.~\eqref{eq:fi_s}. By \cite[Lemma~1.2]{oy}, all these brackets are compatible.
Write $\g_{(t)}$ for the Lie algebra corresponding to $ \{\,\,,\,\}_{t}$. Of course, we merely write $\g$ in
place of $\g_{(1)}$. All Lie algebras $\g_{(t)}$ have the same underlying vector space $\g$.

{\bf Convention~1.}
We often identify $\g$  with $\g^*$  via the Killing form on $\g$. We also think of $\g^*$ as the dual of
any algebra $\g_{(t)}$ and usually omit the subscript `$(t)$' in $\g_{(t)}^*$. However, if $\xi\in\g^*$,
then the stabiliser of $\xi$ in the Lie algebra $\g_{(t)}$ (i.e., with respect to the coadjoint representation
of $\g_{(t)}$) is denoted by $\g_{(t)}^\xi$.

Let $\pi_t$ be the Poisson tensor for $\{\,\,,\,\}_{t}$ and
$\pi_t(\xi)$  the skew-symmetric bilinear form on $\g\simeq T^*_\xi(\g^*)$ corresponding to $\xi\in\g^*$,
cf. Section~\ref{sect:prelim}. A down-to-earth description is that
$\pi_t(\xi)(x_1,x_2)=\{x_1,x_2\}_{(t)}(\xi)$. Set $\rk\pi_t=\max_{\xi\in\g^*}\rk\pi_t(\xi)$.

If $t\ne 0, \infty$, then $\g_{(t)}\simeq \g$ and hence $\ind\g_{(t)}=\ind\g=\rk\g$.

For each Lie algebra
$\g_{(t)}$, there is the related singular set $\g^*_{(t),\sf sing}=\g^*\setminus \g^*_{(t),\sf reg}$\,,
cf.~Eq.~\eqref{eq:regul-set}. Then, clearly,
\[
        \g^*_{(t),\sf sing}=\{\xi\in\g^* \mid \rk \pi_t(\xi)< \rk \pi_t\} ,
\]
which is the union of the symplectic $\g_{(t)}$-leaves in $\g^*$ having a non-maximal dimension.
For aesthetic reasons, we write $\g^*_{\infty,\sf sing}$ instead of $\g^*_{(\infty),\sf sing}$.

Let $\cz_t$ denote the centre of the Poisson algebra $(\gS(\g), \{\,\,,\,\}_{t})$. Formally,
$\cz_t=\gS(\g_{(t)})^{\g_{(t)}}$.
Then $\cz_1=\gS(\g)^{\g}$.
For $\xi\in\g^*$, let $\textsl{d}_\xi F\in\g$ denote the differential of $F\in\gS(\g)$ at $\xi$. It is a standard
fact that, for any $H\in\gS(\g)^\g$, $\textsl{d}_\xi H$ belongs to $\z(\g^\xi)$, where $\z(\g^\xi)$ is the centre of $\g^\xi$.
\\ \indent
Let $\{H_1,\dots,H_l\}$ be a set of homogeneous algebraically independent generators of $\gS(\g)^\g$.
By the {\it Kostant regularity criterion\/} for
$\g$, $\textsl{d}_\xi H_1,\dots,\textsl{d}_\xi H_l$ are linearly independent if and only if $\xi\in\g^*_{\sf reg}$,
see~\cite[Theorem~9]{ko63}. Therefore
\beq     \label{eq:ko-re-cr}
\text{ $\langle \textsl{d}_\xi H_j \mid 1\le j\le l \rangle_{\bbk}=\g^\xi$ \ if and only if \ $\xi\in\g^*_{\sf reg}$.}
\eeq
(Recall that $\g^\xi=\z(\g^\xi)$ if and only if $\xi\in\g^*_{\sf reg}$~\cite[Theorem\,3.3]{p03}.)
For $\xi\in\g^*$, set $\textsl{d}_\xi \cz_t=\left<\textsl{d}_\xi F\mid F\in\cz_t\right>_{\bbk}$.
Then $\textsl{d}_\xi \cz_t \subset \ker \pi_t(\xi)$ for each $t$.
The regularity criterion obviously  holds for any $t\ne 0,\infty$.
That is,
\beq          \label{span-dif}
\text{for }\ t\ne0,\infty, \ \text{ one has  }\  \xi\in\g^*_{(t),\sf reg} \
\Leftrightarrow \
\textsl{d}_\xi \cz_t =\ker \pi_t(\xi) \Leftrightarrow \ \dim \ker \pi_t(\xi)=\rk\g .
\eeq
{\bf Remark.} The same property holds for $t=0$ in some particular cases considered in~\cite[Sections\,4 \& 5]{contr}, which also occur below. For instance, if $(\h,\rr)$ is either $(\be,\ut_-)$, see Section~\ref{sect:b-n},
or $(\be,\g_0)$, see Section~\ref{sect:g0-b} for details.

\subsection{The non-degenerate case}        \label{subs:2Sph}
Let us say that a $2$-splitting is  {\it non-degenerate}, if $\ind\g_{(0)}=\ind\g_{(\infty)}=\rk\g$ and thereby
$\BP_{\sf reg}=\BP$. This is equivalent to that
both subalgebras $\h$ and $\rr$ are spherical, see Theorem~\ref{thm:c=0} and
Remark~\ref{Dima-T}.

Clearly, $\{\cz_t,\cz_{t'}\}_t=0=\{\cz_t,\cz_{t'}\}_{t'}$ for all $t,t'\in\BP$. 
If $t\ne t'$, then each bracket $ \{\,\,,\,\}_s$ is a linear combination of $ \{\,\,,\,\}_{t}$ and $ \{\,\,,\,\}_{t'}$.
Hence $\{\cz_t,\cz_{t'}\}_s=0$ for all $s\in\BP$.  By continuity, this ensures that
$\ker\pi_{t'}(\xi) = \lim_{t\to t'} \textsl{d}_{\xi} \cz_t$ for each $\xi\in\gt g^*_{(t'),\sf reg}$,
cf.~\cite[Appendix]{oy}. Using this one shows that
the centres
$\cz_t$ ($t\in \BP$) generate a {\sf PC}-subalgebra of $\gS(\g)$ with respect to any bracket $\{\,\,,\,\}_t$,
$t\in\BP$. Write $\gZ_{\langle\h,\rr\rangle}:=\mathsf{alg}\langle\cz_t\rangle_{t\in\BP}$ for this subalgebra.
For each $\xi\in\gt g^*$, the space $\textsl{d}_\xi\gZ_{\langle\h,\rr\rangle}$ is the linear span of
$\textsl{d}_\xi \cz_t$ with $t\in\BP$.
In~\cite{bols}, Bolsinov outlined a method for estimating the dimension of such subspaces.
A rigorous presentation is contained in Appendices~\cite{mrl,oy}, which is going to be used in the following proof.

\begin{thm}             \label{thm:dim-Z}
Given a non-degenerate $2$-splitting $\g=\h\oplus\rr$,
\begin{itemize}
\item[\sf (1)] there is a dense open subset $\Omega\in\g^*$ such that $\dim\ker \pi_t(\xi)=\rk\g$ for all
$\xi\in\Omega$ and $t\in\BP$;
\item[\sf (2)] \ for all $\xi\in\Omega$, one has $\dim \textsl{d}_\xi\gZ_{\langle\h,\rr\rangle}=\bb(\g)$ and hence $\trdeg \gZ_{\langle\h,\rr\rangle}=\bb(\g)$.
\end{itemize}
\end{thm}
\begin{proof}
{\sf (1)}  Suppose that
\[
    \xi=\xi_\h+\xi_\rr\in \h^*\oplus\rr^*=\g^*.
\]
The presence of the invertible map $\vp_t$ implies that
$\xi\in\g^*_{(t),\sf sing}$  if and only if \
$\xi_\h+t^{-1}\xi_\rr\in\g^*_{\sf sing}$. Therefore,
\beq      \label{eq-cdt}
   \bigcup_{t\ne 0,\infty} \g^*_{(t),\sf sing} =
   \{\xi_\h+t \xi_\rr \mid \xi_\h+\xi_\rr\in\g^*_{\sf sing}, \, t\ne 0,\infty\} .
\eeq
Since $\codim \g^*_{(t),\sf sing}=3$ for each $t\in\bbk^\times$, the closure
$Y:=\ov{\bigcup_{t\ne 0,\infty} \g^*_{(t),\sf sing}}$ is of codimension $2$ in $\g^*$.  Then we have
$\dim\ker \pi_t(\xi)=\rk\g$  for all $t\in \BP$ and all $\xi$ in the dense open subset
\\       \centerline{
$\Omega=\g^*\setminus (Y\cup \g^*_{(0),\sf sing}\cup \g^*_{\infty,\sf sing})$. }
\\[.5ex]
{\sf (2)} \ By definition, $\textsl{d}_\xi\gZ_{\langle\h,\rr\rangle} =\sum_{t\in\BP}\textsl{d}_\xi\cz
\subset \sum_{t\in\BP} \ker \pi_t(\xi)$. Then~\eqref{span-dif} and the hypothesis  on $\xi$ imply that
$\textsl{d}_\xi\gZ_{\langle\h,\rr\rangle}\supset \sum_{t\ne 0,\infty} \ker \pi_t(\xi)$. Here we have
a $2$-dimensional vector space of skew-symmetric bilinear forms
$a{\cdot}\pi_t(\xi)$ on $\g\simeq T^*_\xi \g^*$, where $a\in\bbk$, $t\in \BP$. Moreover,
$\rk \pi_t(\xi)=\dim\g-\rk\g$ for each $t$.
By~\cite[Appendix]{mrl},  $\sum_{t\ne 0,\infty} \ker \pi_t(\xi)=\sum_{t\in\BP} \ker \pi_t(\xi)$
and $\dim \sum_{t\in\BP} \ker \pi_t(\xi) = \rk\g + \frac{1}{2}(\dim\g- \rk\g)=\bb(\g)$.
\end{proof}

Thus, any non-degenerate $2$-splitting $\g=\h\oplus\rr$ provides a Poisson-commutative subalgebra
$\gZ_{\langle\h,\rr\rangle}\subset \gS(\g)$ of maximal transcendence degree.

Let $\{H_1,\dots,H_l\}$, $l=\rk\g$,  be a set of homogeneous algebraically independent generators of
$\gS(\g)^\g$ and $d_j=\deg H_j$. Recall that for any $H_j$, one has the bi-homogeneous decomposition:
\[
    H_j=\sum_{i=0}^{d_j} (H_j)_{i,d_j-i} ,
\]
and $H_j^\bullet$ is the nonzero bi-homogeneous component of $H_j$ with
maximal $\rr$-degree. Then $\deg_{\rr}\! H_j=\deg_{\rr} H_j^\bullet$.
Similarly, $H_{j,\bullet}$ stands for the nonzero bi-homogeneous component of $H_j$ with
maximal $\h$-degree, i.e., minimal $\rr$-degree.

{\bf Convention~2.}  We tacitly assume that the order of summands in the sum $\g=\h\oplus\rr$ is fixed.
This means that, for a homogeneous $H\in\gS(\g)$, we write $H^\bullet$ (resp. $H_\bullet$) for the bi-homogeneous component of maximal degree w.r.t. the second (resp. first) summand.

It is known that $H_j^\bullet\in \cz\gS(\h\ltimes\rr^{\rm ab})$ and
$H_{j,\bullet}\in \cz\gS(\rr\ltimes\h^{\rm ab})$~\cite[Prop.\,3.1]{coadj}.

\begin{thm}    \label{thm:main3-1}
The algebra $\gZ_{\langle\h,\rr\rangle}$ is generated by $\cz_0$, $\cz_\infty$,  and
the set of all bi-homogeneous components of $H_1,\dots,H_l$, i.e.,
\beq   \label{eq:bihom}
   \{(H_j)_{i,d_j-i} \mid  j=1,\dots,l \ \& \ i=0,1,\dots,d_j\}.
\eeq
\end{thm}
\begin{proof}
Recall that  $\cz(\{\,\,,\,\}_1)=\cz\gS(\g)=\bbk[H_1,\dots,H_l]$. By the definition of
$\{\,\,,\,\}_t$, we have $\cz(\{\,\,,\,\}_t)=\vp_{t}^{-1} (\cz(\gS(\g)))$ for $t\ne 0,\infty$ 
and
\[
    \vp_t(H_j)=(H_j)_{d_j,0}+t (H_j)_{d_j-1,1}+ t^2 (H_j)_{d_j-2,2}+\dots
\]
Using the Vandermonde determinant, we deduce from this that all $(H_j)_{i,d_j-i}$ belong to
$\gZ_{\langle\h,\rr\rangle}$ and the algebra generated by them contains $\cz_t$ with
$t\in \bbk\setminus\{0\}$.
\end{proof}
The main difficulty in applying this theorem is that one has to know the generators of the centres
$\cz_0$ and $\cz_\infty$. The problem is that these centres are not always generated by certain
bi-homogeneous components of $H_1,\dots,H_l$. In the subsequent sections, we consider several nice
examples of non-degenerate $2$-splittings of $\g$, describe the corresponding Poisson-commutative
subalgebras of $\gS(\g)$ and point out some applications to integrable systems.

\section{The Poisson-commutative subalgebra $\gZ_{\langle\be,\ut_-\rangle}$}
\label{sect:b-n}

Let $\g=\ut\oplus\te\oplus\ut_-$ be a fixed triangular decomposition and $\be=\ut\oplus\te$.
The corresponding subgroups of $G$ are $U,T,U_-$, and $B$.
In this section, we take $(\h,\rr)=(\be,\ut_-)$. Then
$\g_{(0)}=\be\ltimes \ut^{\sf ab}_-$ and $\g_{(\infty)}=\ut_-\ltimes \be^{\sf ab}$.
Since $G/U_-$ is quasi-affine, $\ind\g_{(\infty)}=\ind\g$, see Theorem~\ref{thm:c=0}.
By a direct computation, one also obtains  $\ind\g_{(0)}=\ind\g$, cf. Remark~\ref{rem:parab-contr}.
Hence $\g=\be\oplus\ut_-$ is a non-degenerate $2$-splitting.

In order to get explicit generators of the algebra $\gZ_{\langle\be,\ut_-\rangle}$, we first have to
describe the algebras $\cz_0$ and $\cz_\infty$. Recall that $\gS(\g)^\g=\bbk[H_1,\dots,H_l]$ and
$H_i^\bullet$ is the bi-homogeneous component of $H$ of highest degree w.r.t. $\ut_-$. The following
is Theorem~3.3 in~\cite{alafe}.

\begin{prop}   \label{prop:gen-Z0}
For $\g_{(0)}=\be\ltimes\ut^{\sf ab}_-$, the Poisson centre $\cz_0=\cz\gS(\g_{(0)})$  is freely generated
by $H_1^\bullet,\dots, H_l^\bullet$. The bi-degree of $H_j^\bullet$ is $(1,d_j-1)$.
\end{prop}

In our present terminology, one can say that {\bf any} homogeneous generating system
$H_1,\dots,H_l\in\gS(\g)^\g$ is a $\be$-{\sf g.g.s.}

\begin{prop}   \label{prop:gen-Zinf}
For $\g_{(\infty)}=\ut_-\ltimes\be^{\sf ab}$, one has $\cz_\infty=\gS(\te)$, where
$\te\subset\be=\be^{\sf ab}\subset \g_{(\infty)}$.
\end{prop}
\begin{proof}
Since $\be$ is abelian in $\g_{(\infty)}$ and $\be\simeq \g/\ut_-$ as an $\ut_-$-module, we have
$\te\subset \cz_\infty$.  Since
$\ind \g_{(\infty)}=l=\dim\te$, this means that $\gS(\te)\subset\cz_\infty$ is an algebraic extension.
Because $\gS(\te)$ is algebraically closed in $\gS(\g_{(\infty)})$, we conclude that
$\gS(\te)=\cz_\infty$.
\end{proof}

\begin{thm}     \label{thm:b-n_polynomial}
The algebra $\gZ_{\langle\be,\ut_-\rangle}$ is polynomial. It is freely generated by the bi-homogeneous
components $\{(H_j)_{i,d_j-i} \mid 1\le j\le l,\ 1\le i \le d_j-1\}$ and a basis for $\te$.
\end{thm}
\begin{proof}
By Proposition~\ref{prop:gen-Z0}, the generators of $\cz_0$ are certain bi-homogeneous
components of $H_1,\dots,H_l$.
Therefore, combining Theorem~\ref{thm:main3-1}, Proposition~\ref{prop:gen-Z0}, and
Proposition~\ref{prop:gen-Zinf}, we obtain that $\gZ_{\langle\be,\ut_-\rangle}$ is generated by
the bi-homogeneous components of all $H_j$'s and $\gS(\te)$.

The bi-homogeneous component $(H_j)_{d_j,0}$ is the restriction of $H_j$ to
$(\ut_-)^\perp=\be^*\subset\g^*$. (Upon the identification of $\g$ and $\g^*$, this becomes the restriction
to $\be_-=\te\oplus\ut_-$.) As $H_j$ is $G$-invariant, such a restriction depends only on
$\te^*\subset \be^*$; i.e., it is a $W$-invariant element of $\gS(\te)$. Since we already have the whole
of $\gS(\te)$, the functions $\{(H_j)_{d_j,0} \mid 1\le j\le l\}$ are not needed for a minimal generating system.
On the other hand, $(H_j)_{0,d_j}$ is  the restriction of $H_j$ to $\be^\perp=\ut^*_-\simeq \ut$. Therefore,
$(H_j)_{0,d_j}= 0$ for all $j$. Thus, $\gZ_{\langle\be,\ut_-\rangle}$ is generated by the functions pointed out in the statement.
The total number of these generators is $l+\sum_{j=1}^l(d_j-1)=\bb(\g)$. Because
$\trdeg \gZ_{\langle\be,\ut_-\rangle}=\bb(\g)$ (Theorem~\ref{thm:dim-Z}), all these generators are nonzero and algebraically independent.
\end{proof}

\noindent
Thus, we have constructed a polynomial Poisson-commutative subalgebra
$\gZ_{\langle\be,\ut_-\rangle}\subset \gS(\g)$ of maximal transcendence degree.

\begin{thm}     \label{thm:b-n_complete}
The  Poisson-commutative algebra $\gZ_{\langle\be,\ut_-\rangle}$ is complete on every \emph{regular}
coadjoint orbit of $G$.
\end{thm}
\begin{proof}
Given an orbit $G{\cdot}x\subset \g^*_{\sf reg}$,  it suffices to find 
$y\in G{\cdot}x$ such that $\dim\textsl{d}_y \gZ_{\langle\be,\ut_-\rangle}= \bb(\g)$.
Consider first  the regular nilpotent orbit $Ge'$. Let $\{e,h,f\}$ be a principal $\tri$-triple in
$\g$ such that $e\in\ut $, $h\in\te$, $f\in\ut _-$.  Then $y:=e+h-f\in G{\cdot}e'$.
 Here $e\in\gt u_-^*$ and $(h-f)\in\gt b^*$.

We claim that $y\in\g_{(t), {\sf reg}}^*$ for every $t\in\BP$.
Indeed, if $t\ne 0,\infty$, then $te+(h-f)\in\g^*_{\sf reg}$, cf. \eqref{eq-cdt}.
Further, $\g_{(0)}^e=\gt b^e=\g^e$ is commutative and $\dim\g^e=l$.
Therefore also $\g_{(0)}^y=\g^e$ and $y\in\g^*_{(0),{\sf reg}}$.
Finally, $\ad\!^*(\ut _-)(h-f)=\Ann\!(\te\oplus\ut _-)$.
Hence $\dim\g_{(\infty)}^y=\dim\g_{(\infty)}^{h-f}=l$ and $y\in\g^*_{\infty,\sf reg}$.
The claim is settled.

Now we know that $y\in \Omega$, where $\Omega$ is the subset of Theorem~\ref{thm:dim-Z}(1).
By Theorem~\ref{thm:dim-Z}(2), $\dim\textsl{d}_y \gZ_{\langle\be,\ut_-\rangle}= \bb(\g)$.  This means
that $\gZ_{\langle\be,\ut_-\rangle}$ is complete on the regular nilpotent orbit $G{\cdot}e=G{\cdot}e'$, see Lemma~\ref{obvious}.

In general, using the theory of associated cones of Borho and Kraft~\cite{bokr}, one sees that
$Ge\subset \overline{\bbk^\times {\cdot}Gx}$ for any $x\in\g^*_{\sf reg}$.
Since the subalgebra $\gZ_{\langle\be,\ut_-\rangle}$ is homogeneous, we have
\[
\bb(\g)\ge \max_{x'\in Gx}\dim \textsl{d}_{x'} \gZ_{\langle\be,\ut_-\rangle} \ge \max_{e'\in Ge}\dim\textsl{d}_{e'} \gZ_{\langle\be,\ut_-\rangle}=\bb(\g).
\]
The result follows in view of Lemma~\ref{obvious}.
\end{proof}

\begin{rmk}            \label{rem:setting-p}
Our $(\be,\ut_-)$-results can be put in a more general setting in the following way. Let $\p\supset \be$
be a standard parabolic subalgebra with Levi decomposition $\p=\el\oplus\p^{\sf nil}$. This yields the
decomposition $\g=\p^{\sf nil}\oplus \el\oplus\p^{\sf nil}_-$, where $\p_-=\el\oplus\p^{\sf nil}_-$ is the
opposite parabolic. Consider the $2$-splitting $\g=\p\oplus\p^{\sf nil}_-$. Here $\p$ is a spherical
subalgebra, while $\p^{\sf nil}_-$ is spherical if and only if $\p=\be$. Actually,
$c_G(G/P_-^{\sf nil})=\dim \ut(\el)$, where $\ut(\el)=\ut\cap\el$. Then
\[
 \ind\g_{(\infty)}=\ind (\p^{\sf nil}_-\ltimes\p^{\sf ab})=\dim\el ,
\]
cf. Theorem~\ref{thm:c=0}. Moreover, one proves here that $\gZ_\infty= \gS(\el)$, cf.
Proposition~\ref{prop:gen-Zinf}.
Therefore, if $\p\ne \be$, then $\BP_{\sf sing}=\{\infty\}$ and the
{\sf PC}-subalgebra
$\gZ_{\langle\p,\p_-^{\sf nil}\rangle}$ is generated by all $\cz_t$ with $t\ne\infty$. In this case,
$\gZ_{\langle\p,\p_-^{\sf nil}\rangle}\subset \gS(\g)^\el$ and one can prove that
$\trdeg \gZ_{\langle\p,\p_-^{\sf nil}\rangle}=\bb(\g)-\dim \ut(\el)$. To describe explicitly
$\gZ_{\langle\p,\p_-^{\sf nil}\rangle}$, one has to know the structure and generators of
$\cz_0=\cz\gS(\p\ltimes (\p_-^{\sf nil})^{\sf ab})$. However, it is not known whether $\cz_0$ is always polynomial, and generators are only known in some special cases.
For instance, this is so if $\p$ is a minimal parabolic, i.e., $[\el,\el]\simeq\tri$ (see~\cite[Section\,6]{alafe2}). We hope to consider this case in detail in a forthcoming publication.
\end{rmk}

\section{The maximality of  $\gZ_{\langle\be,\ut_-\rangle}$}
\label{sect:5}

\noindent
Here we prove that $\gZ_{\langle\be,\ut_-\rangle}$ is a {\bf maximal} Poisson-commutative
subalgebra of $\gS(\g)$.

Let $\Delta$ be the set of roots of $(\g,\te)$. Then $\g_\mu$ is the root space for $\mu\in\Delta$.
Let $\Delta^+$ be the set of positive roots corresponding to $\ut$. Choose nonzero vectors $e_\mu\in\g_\mu$ and $f_\mu\in\g_{-\mu}$ for any $\mu\in\Delta^+$.
Let  $\ap_1,\ldots,\ap_l$ be the simple roots and $\delta$ the highest root in $\Delta^+$. Write
$\delta=\sum_{i=1}^l a_i \ap_i$ and set $f_i=f_{\ap_i}$.
Assuming that $\deg H_j\le \deg H_i$ if $j<i$ for the basic invariants in $\gS(\g)^\g$, we have
$H_l^\bullet = e_\delta \prod_{i=1}^l f_i^{a_i}$, see~\cite[Lemma~4.1]{alafe}.

Recall that we have two contractions $\g_{(0)}=\be\ltimes\ut^{\sf ab}_-$ and
$\g_{(\infty)}=\ut_-\ltimes\be^{\sf ab}$. As the first step towards proving the maximality of
$\gZ_{\langle\be,\ut_-\rangle}$, we study the subsets $\g^*_{\infty,\sf sing}$ and $\g^*_{(0),\sf sing}$.

\begin{lm}          \label{lm-sing-inf}
{\sf (i)}   $\g^*_{\infty,\sf sing}=\bigcup_{\ap\in\Delta^+}D(\alpha)$,
where $D(\alpha)=\{\xi\in\g^* \mid (\xi,\alpha)=0\}$
and $(\,\,,\,)$ is the Killing form on $\g^*\simeq \g $.
\\ \indent
{\sf (ii)} For any $\alpha\in\Delta^+$ and a generic $\xi\in D(\alpha)$, we have
$\dim\g _{(\infty)}^\xi=l+2$.
\end{lm}
\begin{proof}
For $\xi\in\g^*$, let $C=C_\infty(\xi)$ be the matrix of
$\pi_{\infty}(\xi)|_{\ut _-\times\ut ^{\sf ab}}$. Since $[\gt b,\gt b]_{(\infty)}=0$, we have
$\rk\pi_\infty(\xi)\ge 2\rk C$. Note that if $[e_\ap,f_\beta]_{(\infty)}\ne 0$, then either $\ap=\beta$ or
$\ap-\beta\in\Delta^+$ and therefore $\ap\curge\beta$ in the usual root order ``$\curge$" on $\Delta^+$.
Refining this partial order to a total order on $\Delta^+$  and choosing bases in $\ut $ and $\ut _-$ accordingly, one can bring $C$ into an upper triangular  form with the entries $\xi([f_\ap,e_\ap])$ on the diagonal.
Now it is clear that $\g^*_{\infty,\sf sing}\subset\bigcup_{\ap\in\Delta^+} D(\ap)$.

Let $\xi\in D(\ap)$ be a generic point. Then $\rk C=\dim\ut -1$ and there is a nonzero  $e\in\ut $ such that $\pi_\infty(\xi)(\ut _-,e)=0$. Hence $e\in\g ^\xi_{(\infty)}$.
Because $\te$ is the centre of $\g _{(\infty)}$, we have $\te\subset \g ^\xi_{(\infty)}$ and
$\bb(\g)-1  \ge \rk\pi_{\infty}(\xi)\ge \bb(\g)-2$. Since $\rk\pi_{\infty}(\xi)$ is an even number, it is equal to $\bb(\g)-2$ and therefore $\dim\g ^\xi_{(\infty)}=l+2$. This settles both claims.
\end{proof}

\begin{lm}                  \label{lm-sing-0}
{\sf (i)} Set $D_i=\{\xi\in\g^* \mid \xi(f_i)=0\}$  for $1\le i\le l$.
Then the union of all divisors in  $\g^*_{(0),\sf sing}$ is equal to\/ $\bigcup_{i:\, a_i>1} D_i$.

{\sf (ii)} For any $D_i\subset \g^*_{(0),\sf sing}$\/ and generic $\xi\in D_i$, we have
$\dim\g _{(0)}^\xi=l+2$.
\end{lm}
\begin{proof}
{\sf (i)} \ By~\cite[Theorem\,5.5]{contr}, a fundamental semi-invariant of $\g_{(0)}$ is
$p=\prod_{i=1}^l f_i^{a_i-1}$. The main property of $p$ is that the union of all divisors in
$\g^*_{(0),\sf sing}$ is $\{\xi\in\g^*\mid p(\xi)=0\}$, see~\cite[Def.\,5.4]{contr}. Hence the assertion.

{\sf (ii)} \ Take a generic
$\xi\in D_i\subset \g^*_{(0),\sf sing}$. Then $\xi = y + e$, where $y\in\gt b^*$ and $e\in\ut $ is a subregular nilpotent element
of $\g $, cf.~\cite[Sect.\,5.2]{contr}. According to \cite[Eq.\,(5.1)]{contr},
\[
\dim\g ^\xi_{(0)}=\dim\gt b^e+\ind\gt b^e-l.
\]
On the one side, $\dim(B{\cdot}e)\le \dim\ut -1$, on the other, $\gt b^e\subset\g ^e$ and
$\dim\gt b^e\le l+2$. If $\gt b^e=\g ^e$, then $\ind\gt b^e=l$ \cite[Cor.\,3.4]{p03}, if $\dim\gt b^e=l+1$, then
$\ind\gt b^e\le l+1$. In any case $\dim\gt b^e+\ind\gt b^e \le 2l+2$ and hence
$\dim\g ^\xi_{(0)}=l+2$.
\end{proof}

{\bf Remark.}  Note that all $a_i=1$ if $\g$ is of type {\sf A}. That is, in that case
$\codim \g^*_{(0),\sf sing}\ge 2$.

We will need another technical tool, the pencil of skew-symmetric forms on $\g$ related
to the family $\{\pi_t(\xi)\}_{t\in \bbk\cup \infty}$ for a given $\xi\in\g^*$.  To this end,
we recall some general theory presented in the Appendix to~\cite{oy}.

Let $\eus P$ be a two-dimensional vector space of (possibly degenerate) skew-symmetric bilinear
forms on a finite-dimensional vector space $\gt v$. Set $m=\max_{A\in \eus P }\rk A$, and let
$\eus P_{\sf reg}\subset \eus P$ be the set of all forms of rank $m$. Then $\eus P_{\sf reg}$ is an
open subset of $\eus P$ and $\eus P_{\sf sing}:=\eus P\setminus \eus P_{\sf reg}$ is either $\{0\}$ or a finite union of lines.
For each $A\in \eus P$, let $\ker A\subset \gt v$ be the kernel of $A$.
Our object of interest is the subspace $L:=\sum_{A\in \eus P_{\sf reg}} \ker A$.

\begin{prop}[{cf. \cite[Theorem\,A.4]{oy}}]      \label{prop-JK}
Suppose that $\eus P_{\sf sing}= \bbk C$ with $C\ne 0$ and 
$\rk C=m-2$.
Suppose also that $\rk(A|_{\ker C})=2$ for some $A\in\eus P$.
Then {\sf (1)} $\dim (L\cap \ker C)=\dim\gt v-m$,
$
    {\sf (2)} \ \dim L = \dim \gt v-\frac{m}{2} - 1$, and  {\sf (3)} \ $A(\ker C,  L\cap \ker C)=0$.
\end{prop}
\begin{proof}
The first two assertions are proved in~\cite[Theorem\,A.4]{oy}. We briefly recall the relevant setup.

Take non-proportional $A,B\in  \eus P_{\sf reg}$. By~\cite[Theorem 1(d)]{JK}, there
is the {\it Jordan--Kronecker canonical form\/} for $A$ and $B$. This means that there is a decomposition
$\gt v=\gt v_1\oplus\ldots\oplus \gt v_d$ such that $A(\gt v_i,\gt v_j)=0=B(\gt v_i,\gt v_j)$ for $i\ne j$,
and  the pairs $A_i=A\vert_{\gt v_i}, B_i=B\vert_{\gt v_i}$ have a rather special form.
Namely,  each pair $(A_i,B_i)$ forms either a {\it Kronecker} or a {\it Jordan block\/}
(see \cite[Appendix]{oy} for more details).  Assume that $\dim\gt v_i>0$ for each $i$.

\textbullet\quad For a Kronecker block, $\dim \gt v_i=2k_i+1$, $\rk A_i=2k_i=\rk B_i$ and the same holds
for every nonzero linear combination of $A_i$ and $B_i$.
\\  \indent
\textbullet\quad For a Jordan block,  $\dim \gt v_i$ is even and
both $A_i$ and $B_i$ are non-degenerate on $\gt v_i$. Moreover, there is a unique $\lambda_i\in\bbk$
such that $\det (A_i+\lambda_i B_i)=0$ and hence $\rk(A_i+\lambda_i B_i)\le \dim\gt v_i-2$.
In particular, any Jordan block gives rise to a line $\bbk(A+\lb_i B)\subset \eus P_{\sf sing}$.

Since $\eus P_{\sf sing}$ is a sole line, the critical values $\lb_i$ for all Jordan blocks must be equal.
Furthermore, since $\rk C=m-2$, there must be only one Jordan block, and we may safely assume that
this block corresponds to $\gt v_d$.

Now, we are ready to prove assertion (3).
It is clear that $L\subset \bigoplus_{i<d} \gt v_i$ and
\[
  (L \cap \ker C) \subset \textstyle \bigoplus_{i<d} \ker C_i,
\]
where $\dim\ker C_i=1$ for each $i<d$. Since $A(\gt v_i,\gt v_j)=0$ for $i\ne j$, we obtain
$A(\ker C,L\cap \ker C)=0$.
\end{proof}

Let $\gC\subset\gS(\g)$ be the subalgebra generated by   $\gZ_{\langle\be,\ut_-\rangle}$, $e_\delta$, and $f_i$ with $1\le i\le l$. 
Recall that $H_l^\bullet\in \gZ_{\langle\be,\ut_-\rangle}$ and that $H_l^\bullet = e_\delta \prod_{i=1}^l f_i^{a_i}$ by~\cite[Lemma~4.1]{alafe}. In view of this and Thereorem~\ref{thm:b-n_polynomial},
$\gC$ has a set $\{F_k\mid 1\le k\le \bb(\gt g)+l\}$ of homogeneous generators such that
$\{F_k\mid 1\le k\le l\}$ is a basis of $\gt t$,
$F_k$ is of the form $(H_j)_{i,d_j-i}$ if $l<k<\bb(\gt g)$, and the last $l+1$ elements $F_k$ are root vectors.

By the very construction, we have
\beq             \label{eq-t}
\gZ_{\langle\be,\ut_-\rangle}\subset\gS(\g)^{\te}.
\eeq

\begin{prop}         \label{prop-C}
The subalgebra $\gC$ is algebraically closed in $\gS(\g)$.
\end{prop}
\begin{proof}
For $\gamma\in\g^*$, we set $L(\gamma)= \sum_{t\ne 0,\infty} \ker \pi_t(\gamma)$ and
$V(\gamma)=\textsl{d}_\gamma \gZ_{\langle\be,\ut_-\rangle}$.
If $\gamma\in\g^*_{(t),\sf reg}$ for all $t\ne 0,\infty$, then $L(\gamma)\subset V(\gamma)$ in
view of \eqref{span-dif}.
It follows from \eqref{eq-t} that
$\gamma([V(\gamma),\te])=0$. Consider the following condition on $\gamma$:
\begin{itemize}
\item[($\diamond$)]  \qquad
$\gamma$ is nonzero on at least $l$ elements among
$e_\delta,f_1,\ldots,f_l$.
\end{itemize}
Note that condition~($\diamond$) holds on a big open subset and that the $\te$-weights of the $l$
elements involved, say $x_1,\ldots,x_l$, are
linearly independent. The linear independence of the selected $l$-tuple of $\te$-weights implies that if
$\gamma$ satisfies ($\diamond$), 
$\gamma([\te, x])=0$, and $x\in\left<x_i \mid 1\le i \le l\right>_{\bbk}$, then $x=0$.
Hence here 
$\dim\textsl{d}_\gamma \gC \ge \dim V(\gamma)+l$. In the proof, we compute $\dim\textsl{d}_\gamma \gC$ 
only at points $\gamma$ 
satisfying  ($\diamond$). 

We readily obtain that $\trdeg\gC=\bb(\g)+l$ and hence the homogeneous generators $F_k$ with
$1\le k\le \bb(\gt g)+l$ are algebraically independent. The goal is to show that the differentials of
the polynomials $F_k$ are linearly independent on a {\bf big} open subset. Note that the assertion is
obvious for $\g =\tri$, because here $\gC=\gS(\g)$.

Let $\Omega\subset \g^*$ be the dense open subset defined in Theorem~\ref{thm:dim-Z}.
Then $\dim V(\gamma)=\bb(\g)$ for any $\gamma\in \Omega$. However, the complement of $\Omega$ may contain divisors; i.e., the divisors lying in  $\g^*_{(0),\sf sing}$ or in  $\g^*_{\infty,\sf sing}$, see~\eqref{eq-cdt}.
\\ \indent
\textbullet\quad Concentrate first on the irreducible divisors in $\g^*_{\infty,\sf sing}$. Such
a divisor $D(\ap)$ is the hyperplane defined by $\ap\in \Delta^+$, see Lemma~\ref{lm-sing-inf}{\sf (i)}.
There is a non-empty open subset ${\mathcal U}\subset D(\ap)$ such that any $\tilde\gamma\in {\mathcal U}$  
is regular for all $t\ne \infty$ and satisfies  $\dim\g ^{\tilde\gamma}_{(\infty)}=l+2$, 
see~\eqref{eq-cdt} and Lemmas~\ref{lm-sing-0},~\ref{lm-sing-inf}.
We have $\te\subsetneq \g ^{\tilde\gamma}_{(\infty)}$. Recall from the proof of Lemma~\ref{lm-sing-inf}
that there is a nonzero $e\in \gt u\cap \gt g^{\tilde\gamma}_{(\infty)}$. 
Let $\mu$ be a maximal element in the subset 
$\{\beta\in\Delta^+\mid (e,f_\beta)\ne 0\}$. 
Then $([\gt u_-,e],f_\mu)=0$. Hence $e\in\gt g^{\gamma}_{(\infty)}$ for any 
$\gamma=\tilde\gamma+c f_{\mu}$, where $c\in\bbk$ and $f_\mu$ is  regarded as a linear  function on $\gt g$. 
For $h\in\gt t$ such that $[h,e_{\mu}]=e_{\mu}$, 
we have $\gamma([h,e])=\tilde\gamma([h,e])+c(f_\mu,e)$ and here $(f_\mu,e)\ne 0$. 
For a generic $c\in\bbk$, one obtains $\gamma([\gt t,\gt g^{\gamma}_{(\infty)}])\ne 0$ and 
$\gamma\in{\mathcal U}$. 
On the one hand,  
 $\rk\pi(\gamma)|_{\g ^{\gamma}_{(\infty)}}\ge 2$, on the other hand, $\rk\pi(\gamma)|_{\g^{\gamma}_{(\infty)}}\le 2$
 by \cite[Lemma\,A.3]{oy}.
 According to  \cite[Lemma\,A.1]{mrl}, $L(\gamma)=\sum_{t\ne\infty} \ker \pi_t(\gamma)$.
Now Proposition~\ref{prop-JK} implies  that $\dim L(\gamma)= \bb(\g)-1$ and
\[
    \pi(\gamma)(L(\gamma)\cap \g^{\gamma}_{(\infty)}, \g^{\gamma}_{(\infty)})=0.
\]
By the construction  $\pi(\gamma)(\te,\g^{\gamma}_{(\infty)})\ne 0$. Hence $\te\not\subset  L(\gamma)$ and
$\dim (L(\gamma)+\te)>\dim L(\gamma)$.  For a generic $\gamma\in D(\ap)$, we have then  $\dim V(\gamma)=\bb(\g)$  and
$\dim\textsl{d}_{\gamma} \gC=\bb(\g)+l$.
\\ \indent
\textbullet\quad
Consider a divisor $D_i\subset \g^*_{(0),\sf sing}$ that is defined by
$f_i$ with $a_i>1$, see Lemma \ref{lm-sing-0}. We can safely assume here that $\g$ is not of type
{\sf A}.  Otherwise $\g^*_{(0),\sf sing}$ has no divisors, cf. \cite[Prop.\,4.3]{alafe}.
Because $[\gt b,f_i]_{(0)}\subset \bbk f_i$, we have $f_i\in\g ^{\tilde\gamma}_{(0)}$ for any $\tilde\gamma\in D_i$.
Let $\gamma\in D_i$ be generic.
Lemma~\ref{lm-sing-0} shows that $\rk\pi_0(\gamma)=\dim\g -l-2$ and that $\rk\pi_\infty(\gamma)=l$. 
By \cite[Lemma\,A.1]{mrl}, $L(\gamma)=\sum_{t\ne 0} \ker \pi_t(\gamma)$.
The next task  is to
show that $\gamma$ is nonzero on $[f_i,\g ^\gamma_{(0)}]$. In order to do this, we employ considerations 
from \cite[Sect.~5.2]{contr}. 

Set $\gt p=\gt p_i=\gt b\oplus\bbk f_i$. Then $\overline{(D_i\cap\gt u)}=\gt p^{\sf nil}$ is the nilpotent radical of $\gt p$. 
Write $\gamma=y+e$, where $y\in\gt b^*\simeq\gt b_-$ and $e\in\gt p^{\sf nil}$ is a subregular element of $\gt g$, cf. the proof of Lemma~\ref{lm-sing-0}{\sf (ii)}. We may safely assume that $e$ is a Richardson
element, i.e.,  
$Pe\subset\gt p^{\sf nil}$ is the dense orbit of the parabolic subgroup $P\subset G$ with $\Lie(P)=\gt p$. 
There are two possibilities, either $[\gt p,e]=\gt p^{\sf nil}$ is equal to $[\gt b,e]$ or not. 

Suppose that $\dim[\gt b,e]<\dim\gt p^{\sf nil}$, then there is a nonzero $f\in\gt p^{\sf nil}_-$ such that $(f,[\gt b,e])=0$. At the same time, 
$(f,[\gt p,e])\ne 0$. Therefore $(f,[f_i,e])=([f,f_i],e)\ne 0$. Note that
$\gamma([f,\gt g]_{(0)})=(e,[f,\gt b])=0$, i.e., $f\in\gt g_{(0)}^\gamma$. 
We have also $\gamma([f,f_i])=(e,[f,f_i])\ne 0$. Thus $\gamma$ is nonzero on $[f_i,\gt g^\gamma_{(0)}]$. 

Suppose now that $[\gt b,e]=\gt p^{\sf nil}$. In this case, $\dim\gt b^e=l+1$ and   
$Be$ is dense and open in $\gt p^{\sf nil}$. 
 By \cite[Lemma~5.10]{contr}, $\gt b^e$ is abelian.
Set ${\mathcal U}_0=\{\gamma\in D_i \mid \dim\gt g^\gamma_{(0)}=l+2\}$. 
Then $\gamma=e+y\in {\mathcal U}_0$ for any $y\in\gt b_-$ 
in view of a direct calculation from~\cite[Lemma~4.8]{contr}. 
 Furthermore, $\gamma([f_i,\gt g_{(0)}^\gamma])=0$ if and only if $(f_i,[e+y,\gt g^\gamma_{(0)}])=0$. 
As a point of an appropriate Grassmannian, the subspace $\g^\gamma_{(0)}$ depends on 
$\gamma\in{\mathcal U}_0$ continuously. Therefore it suffices to find just one point $\tilde\gamma\in{\mathcal U}_0$ such that $\tilde\gamma([f_i,\g ^{\tilde\gamma}_{(0)}])\ne 0$. 

Consider first the case, where $[f_i,e_\delta]\ne 0$. Set $\tilde\gamma=e+f_{\delta-\alpha_i}$. 
Then $e_\delta\in\gt g^{\tilde\gamma}$. Here $[f_i,e_\delta]$ is a nonzero scalar multiple of 
$e_{\delta-\alpha_i}$, hence $\tilde\gamma([f_i,e_\delta])=(f_{\delta-\alpha_i},[f_i,e_\delta])\ne 0$.  

In the remaining cases, $(\delta,\alpha_i)=0$, $\gt b^e$ is abelian, and still $a_i>1$. This is possible if and only if $\gt g$ is of type ${\sf B}_l$ with $l\ge 3$ and $i\ge 3$, see \cite{g-co} and \cite[Prop.~5.13]{contr}. 
As a Richardson element in $\p^{\sf nil}$,  we take $e=e_{\alpha_{i-1}+\alpha_i}+\sum_{j\ne i} e_j$; 
next $\beta=\delta-(\alpha_2{+}\alpha_3{+}\ldots{+}\alpha_i)$ and 
$y=f_\beta$. There is a standard choice of root vectors related to elementary skew-symmetric matrices. 
It leads, for example, to $e_{\alpha_{i-1}+\alpha_i}=[e_{i-1},e_i]$. 
After such a normalisation, $\xi:=e_{\beta+\alpha_i}-e_\beta\in\gt b^e$. 
Furthermore, $\ad\!_{(0)}^*(\xi)f_\beta=-[e_\beta,f_\beta]$ and there is 
$$\eta\in\left< f_j, [f_{i-1},f_i] \mid j\ne i-1,i\right>_{\bbk}$$ such that  
$\xi+\eta\in\g^{\tilde\gamma}_{(0)}$ for $\tilde\gamma=e+y$. 
Finally $(e+y, [f_i,\xi+\eta])=(e,[f_i,\eta])+(f_\beta,[f_i,e_{\beta+\alpha_i}])$ 
is nonzero, because $([e,f_i],\eta)=0$ and $([f_\beta,f_i],e_{\beta+\alpha_i})\ne 0$.

Now we know that $\pi(\gamma)(f_i,\g ^\gamma_{(0)})\ne 0$. By \cite[Lemma\,A.3]{oy},
$\rk(\pi(\gamma)|_{\g ^\gamma_{(0)}})\le 2$, hence the rank in question is equal to $2$.
According to Proposition~\ref{prop-JK},  $\dim L(\gamma) = \bb(\g)-1$
and  $f_i\not\in L(\gamma)$.

Note that  $\pi(\gamma)(f_i,\te)=0$. Furthermore, if $x\in\left< e_\delta, f_j \mid j\ne i \right>_{\bbk}$
and $\pi(\gamma)(\te,x)=0$, then $x=0$. Therefore $\dim\textsl{d}_\gamma \gC=\bb(\g)+l$.
Since $\textsl{d}_\gamma \gC=\left<\textsl{d}_\gamma F_k \mid 1\le k\le \bb(\gt g)+l\right>_{\bbk}$,
the goal is achieved, the differentials $\textsl{d}F_k$
are linearly independent on a big open subset.
According to  \cite[Theorem \,1.1]{ppy}, the subalgebra $\gC$ is algebraically closed in $\gS(\g)$.
\end{proof}

\begin{thm}                   \label{max-u}
The algebra  $\gZ_{\langle\be,\ut_-\rangle}$ is a maximal Poisson-commutative
subalgebra of\/ $\gS(\g)$.
\end{thm}
\begin{proof}  Let $\gA\subset \gS(\g)$ be a Poisson-commutative subalgebra and
$\gZ_{\langle\be,\ut_-\rangle}\subset \gA$. Since
$\trdeg \gZ_{\langle\be,\ut_-\rangle}=\bb(\g)=\trdeg\gA$, each element $x\in\gA$ is algebraic over
 $\gZ_{\langle\be,\ut_-\rangle}$. Hence it is also algebraic over $\gC$ and by Proposition~\ref{prop-C},
we have $x\in\gC$.  Since $\te\subset  \gZ_{\langle\be,\ut_-\rangle}$,  we have $\{\te,x\}=0$. The
algebra of $\te$-invariants in
$\gC$ is generated by  $\gZ_{\langle\be,\ut_-\rangle}$ and the monomials
$e_\delta^c f_1^{c_1}\ldots f_{l}^{c_l}$ such that $c_i=ca_i$.
Each such monomial is a power of $H_l^\bullet$. Therefore $x\in  \gZ_{\langle\be,\ut_-\rangle}$ and
$\gZ_{\langle\be,\ut_-\rangle}=\gA$.
\end{proof}

\section{The Poisson-commutative subalgebra $\gZ_{\langle\be,\g_0\rangle}$}
\label{sect:g0-b}

If $\sigma$ is an involution of $\g$, then $\g=\g_0\oplus \g_1$, where
$\g_i=\{x\in\g\mid \sigma(x)=(-1)^ix\}$. As is well-known, $\g_0$ is a spherical subalgebra of $\g$.
Therefore, there is a Borel subalgebra $\be$ such that $\g_0+\be=\g$.

An involution $\sigma$ is said to be of {\it maximal rank}, if $\g_1$ contains a Cartan subalgebra of $\g$. Then $\dim \g_1=\dim \be$, $\dim\g_0=\dim\ut$, and such $\sigma$ is unique up to $G$-conjugation.
Therefore, in the maximal rank case, there is a Borel subalgebra $\be$ such that
\beq     \label{eq:direct-b-g0}
   \be\oplus\g_0=\g .
\eeq
Recall that (for $\bbk=\BC$) there is a bijection between the (conjugacy classes of complex) involutions of $\g$
and the real forms of $\g$, see e.g.~\cite[Ch.\,4,\ 1.3]{t41}. Under this bijection the involution of
maximal rank corresponds to the split real form of $\g$. This bijection also allows us to associate the Satake diagram~\cite[Ch.\,4,\ 4.3]{t41} to any involution.
In this section, we assume that $\sigma$ is of maximal rank and take $(\h,\rr)=(\be,\g_0)$ such that
Eq.~\eqref{eq:direct-b-g0} holds. As in Section~\ref{sect:b-n}, to describe the generators of
$\gZ_{\langle\be,\g_0\rangle}$, we need a set of generators for the  Poisson centres $\cz_0=\cz\gS(\be\ltimes \g_0^{\sf ab})$ and $\cz_\infty=\cz\gS(\g_0\ltimes\be^{\sf ab})$.
By the Independence Principle of Section~\ref{sect:2}, we have
\[
\text{$\be\ltimes \g_0^{\sf ab}\simeq \be\ltimes\ut^{\sf ab}_-$ \quad and \quad
 $\g_0\ltimes\be^{\sf ab}\simeq \g_0\ltimes\g_1^{\sf ab}$.}
\]
Hence the structure of $\cz_0$ is already described in Prop.~\ref{prop:gen-Z0},
whereas the Poisson centre of $\gS(\g_0\ltimes\g_1^{\sf ab})$ is described in \cite{coadj}. Namely,
$\cz\gS(\g_0\ltimes\g_1^{\sf ab})$ is freely generated by the bi-homogeneous components of $\{H_i\}$
of minimal degree w.r.t. $\g_0$, i.e., of maximal degree w.r.t. $\g_1$ (or $\be$). In particular,
any generating system $H_1,\dots,H_l\in \gS(\g)^\g$ is a $\g_0$-{\sf g.g.s.}

\begin{thm}      \label{thm:g0-b-polynomial}
The algebra $\gZ_{\langle\be,\g_0\rangle}$ is polynomial. It is freely generated by the bi-homogeneous
components $\{(H_j)_{i,d_j-i} \mid 1\le j\le l,\ 1\le i \le d_j\}$.
\end{thm}
\begin{proof}
It follows from the above discussion and Theorem~\ref{thm:main3-1} that $\gZ_{\langle\be,\g_0\rangle}$
is generated by the bi-homogeneous components of all $\{H_i\}$. The total number of all
bi-homogeneous components equals $\sum_{j=1}^l(d_j+1)=\bb(\g)+l$. As in the proof of
Theorem~\ref{thm:b-n_polynomial}, the component $(H_j)_{0,d_j}$ is the restriction of $H_j$ to
$\be^\perp$. Under the identification of $\g$ and $\g^*$, we have $\be^\perp=\ut$. Therefore
$(H_j)_{0,d_j}\equiv 0$ for all $j$. Thus, there remain
at most $\bb(\g)$ nonzero bi-homogeneous components and, in view of Theorem~\ref{thm:dim-Z},
these components must be nonzero and algebraically independent.
\end{proof}

\noindent
Thus, we have obtained a polynomial Poisson-commutative subalgebra $\gZ_{\langle\be,\g_0\rangle}$
of $\gS(\g)$ of maximal transcendence degree.

\begin{ex}    \label{ex:sl-so}
{\sf (1)} \ If $\g=\sln$ and  $\sigma$ is of maximal rank, then $\g_0=\son$.  Here
$\dim\g^*_{(0),\sf sing}\le \dim\g  - 2$ by \cite[Theorem\,3.3]{coadj} and  $\dim\g^*_{\infty,\sf sing}\le \dim\g -2$
by~\cite[Section\,4]{alafe}. In view of \eqref{eq-cdt}, this implies that the open subset  $\Omega$ of
Theorem~\ref{thm:dim-Z} is big. Thus, the differentials of the free generators of
$\gZ_{\langle\be,\g_0\rangle}$ are linearly independent on the big open subset $\Omega$.
By \cite[Theorem \,1.1]{ppy}, this means that $\gZ_{\langle\be,\g_0\rangle}$ is an algebraically closed subalgebra of
$\gS(\g)$. Since $\trdeg \gZ_{\langle\be,\g_0\rangle}$ is maximal possible among all
{\sf PC}-subalgebras, $\gZ_{\langle\be,\g_0\rangle}$ is a {\bf maximal} {\sf PC}-subalgebra of
$\gS(\sln)$.

{\sf (2)} \ By~\cite[Theorem\,4.4]{alafe}, if $\g$ is simple, but $\g\ne\sln$, then $\dim\g^*_{\infty,\sf sing}= \dim\g -1$. Therefore, the above argument does not generalise. Still, this does not prevent
$\gZ_{\langle\be,\g_0\rangle}$ from being a maximal {\sf PC}-subalgebra. Actually,
we do not know yet whether $\gZ_{\langle\be,\g_0\rangle}$ is maximal for the other simple $\g$.
 \end{ex}

\begin{rmk}   \label{rmk:any-invol}
If $\sigma$ is not of maximal rank, then $\dim\g_0>\dim\ut$ and the sum $\g_0+\be=\g$  {\bf cannot} be
direct. Given $\g_0$, one can choose a generic ``opposite'' Borel subalgebra $\be$ such that
$\dim(\be\cap\g_0)$ is minimal possible and $\be\cap\g_0$ is closely related to
a Borel subalgebra of a certain Levi subalgebra. Namely,
there is a parabolic subalgebra $\p\supset\be$, with the standard Levi subalgebra $\el\subset\p$, such
that $[\el,\el]\subset \p\cap \g_0\subset \el$ and $\be\cap\g_0$ is a Borel subalgebra of
$\es:=\p\cap \g_0$~\cite[Chapters\,1,\,2]{these-p}. (The semisimple algebra $[\el,\el]$ corresponds to
the subset of black nodes of the Satake diagram of $\sigma$.) Therefore, there is always a {\it solvable\/}
subalgebra $\h\subset\be$ normalised by $\te$ such that $\h\oplus (\g_0\cap\be)=\be$ and hence
$\h\oplus\g_0=\g$. Here $\p^{\sf nil}\subset\h\subset \p^{\sf nil}\oplus\z(\el)$, where
$\z(\el)$ is the centre of $\el$.
Hence {\it\bfseries any} involution $\sigma$ gives rise to a natural $2$-splitting of $\g$.
But this $\h$ not necessarily spherical. A sufficient condition for sphericity is that
the Satake diagram of $\sigma$ has no black nodes. (This
is equivalent to that $\g_1\cap\g_{\sf reg}\ne\varnothing$.) Then $\p=\be$ and $\be\cap\g_0\subset \te$.
Hence $\h\supset\ut=\be^{\sf nil}$ and thereby $\h$ is spherical.
Thus any involution of $\g$ having the property that $\g_1\cap\g_{\sf reg}\ne\varnothing$ gives rise to a {\sl non-degenerate} $2$-splitting.
\\ \indent \textbullet \ If $\g$ is simple, then such involutions that are not of maximal rank exist only for
$\GR{A}{n}$, $\GR{D}{2n+1}$, and $\GR{E}{6}$.
However, it is not yet clear how to describe explicitly the Poisson centre
$\cz\gS(\h\ltimes\g_0^{\sf ab})$ if \ $\h\ne\be$.
\\ \indent \textbullet \ Yet another similar possibility is the semisimple algebra
$\g\oplus\g\simeq\g\times\g$, where $\g$ is simple and $\sigma$ is the permutation of summands.
Here everything can be accomplished explicitly, see the following section.
\end{rmk}

\section{Poisson-commutative subalgebras related to a $2$-splitting of $\g\times\g$}
\label{sect:k-b}

\noindent
In this section, we consider in detail the good case mentioned at the end of Remark~\ref{rmk:any-invol}
and its application to Lie algebras over $\BR$.

Let $\tau$ be the involution of $\tilde\g:=\g\oplus\g\simeq\g\times\g$ such that $\tau(x_1,x_2)=(x_2,x_1)$.
Then $\tilde\g_0=\Delta_\g\simeq\g $ is the usual diagonal in $\g\times\g$ and $\tilde\g_1$
is the antidiagonal $\Delta^{(-)}_\g=\{(x,-x)\mid x\in\g\}$. Here a generic opposite Borel subalgebra of
$\tilde\g$ for $\Delta_\g$ is $\be\times\be_-$ and $\Delta_\g\cap(\be\times\be_-)=\Delta_\te$.
It follows that a complementary solvable subalgebra for $\Delta_\g$ is
\[
   \h=\Delta^{(-)}_\te\oplus (\ut\times\ut_-) ,
\]
where $\Delta^{(-)}_\te=\{(x,-x)\mid x\in\te\}$.
This yields the $2$-splitting
\beq    \label{eq:2-twisted}
  \g\times\g=\h\oplus\Delta_\g ,
\eeq
associated with $\tau$ in the sense of Remark~\ref{rmk:any-invol}.
Next step is to prove that this $2$-splitting is non-degenerate and both related {\sf IW}-contractions
of $\g\times\g$ have a polynomial ring of symmetric invariants.

By the Independence Principle,
the {\sf IW}-contraction $(\g\times\g)_{(\infty)}=\Delta_\g\ltimes\h^{\sf ab}$ is isomorphic to the {\it Takiff Lie algebra\/} $\g\ltimes\g^{\sf ab}$.
A description of the symmetric invariants of $\g\ltimes\g^{\sf ab}$ is due to Takiff~\cite{takiff},
cf. also \cite{p05}. This implies that there is a good generating system here. More explicitly,
let $\{H_{j,I}, H_{j,II}\mid 1\le j\le l\}$ be the obvious set of
basic symmetric invariants of $\g\times\g$. Then $\{H_{j,I} \pm H_{j,II}\mid 1\le j\le l\}$ is 
a $(\Delta_\g)$-{\sf g.g.s.} 

Set $\BV=\Delta_\te\oplus(\ut _-\times\ut)\subset \g\times\g$.
As $\BV$ is a complementary space to
$\h$ in $\g\times\g$, it follows from the independence principle that
\[
   \q:= \h\ltimes \BV^{\sf ab}\simeq \h\ltimes\Delta_\g^{\sf ab}= (\g\times\g)_{(0)},
\]
i.e., $\q$ is the {\sf IW}-contraction of $\g\times\g$ associated with $\h$. Recall that $d_j=\deg H_j$.

\begin{prop}                \label{prop-7-0}
We have $\ind\q=2l$ and\/ $\gS(\q)^{\q}$ is freely generated by the polynomials
\[
  \boldsymbol{F}_j=(H_{j,I}- H_{j,II})^\bullet \ \text{ with } \ 1\le j\le l
\]
and a basis of \, $\Delta_\te$.
\end{prop}
\begin{proof}
Let $\gamma=(\xi,\xi)\in\q^*$ be a linear form such that $\xi\in\te^*$ and $\dim\g^\xi=l$.
Then $\dim\q^\gamma = 2l$ and thereby $\ind\q\le 2l$. Note also that $\Delta_\te$ belong to the centre of $\q$, i.e.,
$[\Delta_\te,\g\times\g]_{(\infty)}=0$.

Let $F_j=F_{j,I}$ be the highest  $\ut _-$-component of $H_j\in\gS(\g)^{\g }$ w.r.t. the splitting
$\g=\be\oplus \ut_-$.
By~\cite[Lemma\,5.7]{contr}, we have $F_j\in\gS(\ut \oplus\ut _-)$.
Recall that by \cite{alafe} the elements  $\{F_j\}$ are algebraically independent  and
$\deg_{\ut _-} F_j=d_j-1$. Similarly, let $F_{j,II}$ be the highest  $\ut$-component of $H_{j,II}$ w.r.t. the splitting
$\g=\be_-\oplus \ut$.

For each $j\in\{1,\dots,l\}$ and $s\in\{I,II\}$, we have  $H_{j,s}^\bullet \in\gS(\Delta_\te)$.
In view of this and the above paragraph,
\[
  \boldsymbol{F}_j= F_{j,I} -  F_{j,II} + \tilde F_j, \ \text{ where } \ \tilde F_j \in \Delta_\te\gS(\q).
\]
We see that $\gS(\q)^{\q}$ contains $2l$ algebraically independent elements and hence
$\ind\q\ge 2l$. Thereby  $\ind\q=2l$.

Assume that $\dim(\left<\textsl{d}_\gamma \boldsymbol{F}_j \mid 1\le j\le l\right>_{\bbk}+\Delta_\te)<2l$
for all points $\gamma$ of a divisor $D\subset\q^*$. Then $D$ is defined by a homogeneous polynomial
and $\dim(D\cap\Ann(\Delta_\te))\ge \dim\q-l-1$.
If we write $\gamma=\gamma_I+\gamma_{II}$ for $\gamma\in (D\cap\Ann(\Delta_\te))$,  then the
differentials $\{\textsl{d}_{\gamma_s} F_{j,s}\}$ are linearly dependent at $\gamma_s$ and thus
\[
   \gamma_I \in (\be\ltimes\ut^{\sf ab}_-)^*_{\sf sing}, \quad \gamma_{II} \in (\be_-\ltimes\ut^{\sf ab})^*_{\sf sing}
\]
by \cite{contr}.
The intersection $(\gt b\ltimes\ut _-^{\sf ab})^*_{\sf sing} \cap \Ann (\te)$ is a proper closed subset of
$\Ann\!(\te)$, cf. Lemma~\ref{lm-sing-0}{\sf (i)}.  Thereby $\dim(D\cap\Ann (\Delta_\te))\le \dim\q-l-2$,
a contradiction.

By~\cite[Theorem \,1.1]{ppy}, the polynomials $\{\boldsymbol{F}_j\}$, together with a basis of
$\Delta_\te$, generate an algebraically closed subalgebra of $\gS(\q)$. Since $\ind\q=2 l$, we are done.
\end{proof}

Thus, the above results show that $2$-splitting~\eqref{eq:2-twisted} is non-degenerate, and we can
consider the corresponding Poisson-commutative subalgebra of $\gS(\g\times\g)$.

\begin{thm}    \label{thm:g-h}
The algebra $\gZ_{\langle\h,\Delta_\g\rangle}\subset\gS(\g\times\g)$ is polynomial.
It is freely generated by the bi-homogeneous components
\[
   \{(H_{j,I} + H_{j,II})_{s,d_j-s},  (H_{j,I} - (-1)^{d_j}H_{j,II})_{s',d_j-s'}
   \mid 1\le j\le l,\ 1\le s \le d_j,  \,  1\le s' \le d_j-1 \}
\]
together with a basis for\/ $\Delta_\te$.
\end{thm}
\begin{proof}
It follows from the description of $\cz_{\infty}$ and Theorem~\ref{thm:main3-1} that
$\gZ_{\langle\h,\Delta_\g\rangle}$ is generated by all the bi-homogeneous components of
$\{H_{j,I} \pm H_{j,II}\}$ and $\cz_0$.
By Proposition~\ref{prop-7-0}, $\cz_0$ is generated by the bi-homogeneous components of the
form $\{(H_{j,I}- H_{j,II})^\bullet\}$ with $j=1,\dots,l$ and
a basis for $\Delta_\te$. Thus, the total number of generators is at most $2\bb(\g)+3l$.

Since the components of the form $(H_{j,I} \pm H_{j,II})_{0,d_j}$ are either zero or belong to
$\gS(\Delta_\te)$, they are redundant.  Notice also that $(H_{j,I} - (-1)^{d_j}H_{j,II})_{d_j,0}=0$ for all $j$.
Therefore, there are at most $2\bb(\g)=\bb(\g\times\g)$ nonzero generators and, in view of
Theorem~\ref{thm:dim-Z}, they must be algebraically independent.
\end{proof}

\subsection{The real picture}
Assume now that $\bbk=\BC$. Let $\ka$ be a compact real form of $\g$. Then $\be\cap \ka=i\te_{\BR}$,
where $\te_{\BR}$ is a maximal torus in a split real form $\g _{\BR}\subset \g$.
Set $\rr=\te_{\BR}\oplus\ut$. It is an $\BR$-subalgebra of $\be$ and we have the real $2$-splitting
$\g =\rr\oplus\ka$, which is the Iwasawa decomposition of $\g$ as a real Lie algebra.
The complexification of this decomposition is conjugate to the
 $2$-splitting of $\g\times\g$ defined by  Eq.~\eqref{eq:2-twisted}.
Here $(\g,\ka, \rr)$ is a Manin triple over $\BR$, see~\cite[Sect.~5.3]{duzu}.

We choose the basic symmetric invariants of $\g$ such that each $H_j$ takes only real values on
$\te_{\mathbb R}$. Over $\mathbb R$, $\gS(\g)^{\g }$ is generated by $\Re H_j$ and $\Im H_j$ with
$1\le j\le l$.

Translating Theorem~\ref{thm:g-h} to the real setting, we obtain the following result.

\begin{thm}    \label{thm:k-b-polynomial}
Let $\gS_\BR(\g)$ be the symmetric algebra over $\BR$ of the real Lie algebra $\g$. Then
the $\BR$-algebra $\gZ_{\langle\rr,\gt k\rangle}\subset \gS_\BR(\g)$ is polynomial. It is freely generated by the bi-homogeneous
components
\[
   \{(\Re H_j)_{s,d_j-s},  (\Im  H_j)_{s',d_j-s'}   \mid 1\le j\le l,\ 1\le s \le d_j,  \,  1\le s' \le d_j-1 \}
\]
together with a basis of \ $i\te_{\mathbb R}$.
\end{thm}

\begin{rmk}    \label{rem:real-general}
We associated a $2$-splitting of $\g$ to any (complex) involution $\sigma$, see
Remark~\ref{rmk:any-invol}.  If $\g_{\BR,\sigma}$ is the real form of $\g$ corresponding to $\sigma$, then
the Iwasawa decomposition of $\g_{\BR,\sigma}$ is just the real form of that $2$-splitting.
We hope to elaborate on this relationship and related {\sf PC}-subalgebras in the future.
\end{rmk}

\end{document}